\documentclass[fleqn,leqno,abstract=true]{scrartcl}

\usepackage{mathtools}
\mathtoolsset{showonlyrefs=true}
\usepackage{amssymb,amsthm,libertine,libertinust1math}
\usepackage{verbatim}
\usepackage[pdftex]{graphicx}
\usepackage{subcaption}

\usepackage{todonotes}
\usepackage{bm}

\usepackage[onehalfspacing]{setspace}
\KOMAoptions{DIV=classic}

\usepackage{xcolor}
\definecolor{webgreen}{rgb}{0,.5,0}
\definecolor{webbrown}{rgb}{.8,0,0}
\definecolor{emphcolor}{rgb}{0.5,0.95,0.95}

\usepackage{hyperref}
\hypersetup{%
	colorlinks=true,
	linkcolor=webbrown,
	filecolor=webbrown,
	citecolor=webgreen,
	breaklinks=true,
	pdfstartview=FitH, %%Fit, FitB, FitH
	bookmarksopen=true,
	bookmarksnumbered=true
}

\DeclarePairedDelimiterX{\innp}[1]{\langle}{\rangle}{#1}
\DeclarePairedDelimiterX{\floor}[1]{\lfloor}{\rfloor}{#1}
\DeclarePairedDelimiterX{\ceil}[1]{\lceil}{\rceil}{#1}
\let\abs\relax\DeclarePairedDelimiterX{\abs}[1]{\lvert}{\rvert}{#1}
\let\norm\relax\DeclarePairedDelimiterX{\norm}[1]{\lVert}{\rVert}{#1}
\DeclareMathOperator{\diag}{diag}

\newcommand{\Id}{\mathsf{I}}
\DeclareMathOperator{\coal}{Coal}
\DeclareMathOperator{\muta}{Dice}
\newcommand {\E}{\mathbb{E}}
\numberwithin{equation}{section}

\newtheorem{theorem}{Theorem}[section]
\newtheorem{proposition}{Proposition}[section]

\newtheorem{remark}{Remark}[section]
\newtheorem{lemma}{Lemma}[section]

\newtheorem{definition}{Definition}[section]

\numberwithin{remark}{section} 
\numberwithin{proposition}{section}
\numberwithin{corollary}{section}

\newcommand{\red}{\textcolor[rgb]{1.00,0.00,0.00}}

\newcommand{\cyan}{\textcolor[rgb]{0.00,0.71,0.92}}

\def\@MRExtract#1 #2!{#1}     % thanks, Martin!
\newcommand{\MR}[1]{% we need to strip the "(...)"
  \xdef\@MRSTRIP{\@MRExtract#1 !}%
  \href{https://mathscinet.ams.org/mathscinet-getitem?mr=\@MRSTRIP}{MR\@MRSTRIP}}

\newcommand{\EMAIL}[1]{\href{mailto:#1}{#1}}

\title{Partially exchangeable Markov chains and characterisation of multitype Lambda-coalescents}
\author{%
    Adri\'an Gonz\'alez Casanova\footnote{School of Mathematics and Statistical Sciences, Arizona State University, Tempe, USA. \EMAIL{agonz591@asu.edu}}
    \and
    Noemi Kurt\footnote{Institut f\"ur Mathematik, Goethe-Universit\"at Frankfurt, Robert-Mayer-Str. 10, 60325 Frankfurt am Main, Germany. \EMAIL{kurt@math.uni-frankfurt.de}}
    \and
    Imanol Nu{\~n}ez\footnote{Department of Probability and Statistics, Centro de Investigaci\'on en Matem\'aticas A.C. Calle Jalisco s/n. C.P. 36240, Guanajuato, Mexico. \EMAIL{imanol.nunez@cimat.mx, jluis.garmendia@cimat.mx}}
    \and
    Jos\'e-Luis P\'erez\footnotemark[3]
}%AUTHORS
\date{\today}

\begin{document}

\maketitle

\begin{abstract}
	In this paper, we study consistent and partially exchangeable sequences of Markov chains on a finite state space. We provide a characterisation of the admissible transition rates via  a decomposition into individual and coordinated motion of particles. As a consequence, we find a characterisation of multitype Lambda-coalescents with multiple switches. Moreover, we provide convergence and duality results for the corresponding process of limiting relative frequencies that we call the de Finetti measure process, and discuss a number of examples from the recent literature.
\end{abstract}

\section{Introduction}

Since the seminal work of de Finetti on \emph{exchangeability} (see \cite{definettiTheoryProbabilityCritical2017}), culminating in the celebrated de Finetti representation theorem, the concept of exchangeability has garnered significant attention, both in probability (e.g. \cite{aldousExchangeabilityRelatedTopics1985}) and in Bayesian statistics (e.g. \cite{phadiaPriorProcessesTheir2016}).
In essence, the de Finetti representation theorem states that any (countable) collection of random variables whose distribution is invariant under permutation of finitely many of their indices is conditionally independent given some underlying random measure, which we refer to as the \emph{de Finetti measure}.

While originally formulated for sequences of binary random variables, the representation theorem has since been extended to random variables taking values in Polish spaces---topological spaces that are separable and completely metrizable.
Moreover, the core concept of exchangeability has been generalized in various directions, including the notion of \emph{partial exchangeability} (which was already considered by de Finetti).
We refer the reader to \cite{aldousExchangeabilityRelatedTopics1985,kallenbergProbabilisticSymmetriesInvariance2005} for thorough accounts on this topic.

The representation theorem mentioned in the previous paragraph is often stated in an abstract sense.
As observed by Diaconis and Freedman in \cite{diaconis11FinettisGeneralizations2023,diaconisRecentProgressFinettis1988}, from a Bayesian perspective, such an abstract representation does not allow one to meaningfully specify a prior distribution on the probability measures arising in it.
Thus, one might impose certain restrictions to reduce the space of admissible probability measures.
Nonetheless, properties of specific collections of exchangeable random variables, such as continuous-time Markov chains, remain underexplored.
In this paper, we study a particular system of Markov chains arranged in a partially exchangeable sequence, evolving on a finite state space.
This sequence satisfies a consistency condition: any finite subcollection of chains is again a Markov chain.
One of our main results is a characterization of all such sequences.

We caution the reader, as partially exchangeable sequences of Markov chains share some similarities with random walks that have exchangeable increments as studied by Freedman in \cite{diaconisFinettisTheoremMarkov1980,freedmanFinettisTheoremContinuous1996}.
However, the frameworks differ significantly: Freedman's results concern mixtures of Markov chains yielding a single process, while we analyze sequences of Markov chains.

This type of structure also arises in population genetics, particularly in the study of genealogical processes.
Specifically, random exchangeable partitions play a central role in coalescent theory, underpinning models such as the $\Lambda$-coalescent and the $\Xi$-coalescent \cite{berestyckiRecentProgressCoalescent2009}.
Recently, with the introduction of multitype coalescent processes such as the multitype $\Lambda$-coalescent \cite{johnstonMultitypeLcoalescents2023}, a class of $\Xi$-coalescents in \cite{mohleMultitypeCanningsModels2024} and another in \cite{daipraMultitypeXcoalescentsStructured2025} respectively, and the multitype coalescent in \cite{flammGenealogyMultitypeCannings2025}, partial exchangeability has emerged as a key concept in population genetics.

The framework of partially exchangeable sequences of Markov chains provides a formal model for coordinated mutation, where individuals of various types may undergo simultaneous mutations, with the distribution of the resulting type of each individual determined by their initial type.
On one hand, this extends the by now classical $\Lambda$-coalescent models to the case of coordinated mutation, but more generally gives insight into the interplay of consistency, exchangeability and coordination.
See also \cite{johnstonMultitypeLcoalescents2023}, \cite{gonzalezcasanovaMultitypeLcoalescentsContinuous2024}, \cite{gonzalezcasanovaParticleSystemsCoordination2021}, \cite{flammGenealogyMultitypeCannings2025} and \cite{daipraMultitypeXcoalescentsStructured2025}.

We call the processes we investigate here \emph{dice processes}, because we construct them as a collection of $n\in\mathbb{N}$ Markov chains that evolve on a finite state space of cardinality $d\in \mathbb{N},$ where we imagine many $d$-faceted dice that determine the jumps of the processes according to certain rules.
The precise set of admissible rules in this setup leading to partially exchangeable, consistent processes is characterized in our first main result.

We note that related results have appeared in \cite{carinciConsistentParticleSystems2021}, where a similar class of processes is studied from the perspective of configuration processes, with consistency defined over randomly chosen subsystems.
In contrast, our framework operates at the level of coordinate processes, emphasizing partial exchangeability alongside consistency, allowing deterministic selection of subsystems.
This highlights the broader relevance of the structural principles underlying dice processes.

Once we have characterized dice processes as partially exchangeable and consistent processes, we will use them to model coordinated mutation in genealogical processes.
This leads to an extension of the multitype-$\Lambda$-coalescent introduced in \cite{johnstonMultitypeLcoalescents2023} that now includes genealogies with multiple (or coordinated) switching, such as the seedbank coalescent with multiple switching studied in \cite{Blathetal2020}.
Namely, in our second main result we characterize the multitype coalescents in a \emph{Pitman sense}.
That is, we characterize multitype coalescent processes that are consistent, partially exchangeable, and such that \emph{no other event may take place when a merger occurs}.

It is worth noting that coordinated mutation is implicitly present in the multitype coalescent process presented in \cite{flammGenealogyMultitypeCannings2025}.
They give integral representations for rates governing various kinds of coalescent mechanisms---where multiple coalescent events may occur simultaneously.
In contrast, we provide an integral representation of the mechanism governing coordinated mutation, where no coalescence occurs.

Beyond the setting of genealogies, dice processes also offer a natural framework for modeling the microscopic evolution of mass in certain interacting particle systems.
In particular, we construct exchangeable dice processes whose de Finetti measure processes evolve according to a class of dynamics we refer to as \emph{stochastic exchange models}, adopting the terminology of \cite{kimSpectralGapKMP2025}.
These include the \emph{averaging process} and its various extensions, which have recently drawn considerable attention.

The averaging process was introduced in \cite{aldousLectureAveragingProcess2012} as a simple, yet at the time largely unexplored, example of an interacting particle system evolving on a continuous state space---specifically, the space of probability measures on a finite set.
It is a pure jump Markov process in which, at each jump time, two coordinates are selected at random and the mass assigned to them is redistributed equally.
By construction, the uniform distribution is invariant under the dynamics.

The convergence behavior of the averaging process has been studied in various settings.
In \cite{chatterjeePhaseTransitionRepeated2022}, it was shown that the process exhibits the cut-off phenomenon when the interacting coordinates are selected uniformly at random.
This result was extended in \cite{quattropaniMixingAveragingProcess2023}, where the redistribution follows fixed non-uniform weights.
Further generalizations appear in \cite{caputoRepeatedBlockAverages2024}, where more than two coordinates may be selected and the mass is averaged proportionally to the number of participants; and in \cite{kimSpectralGapKMP2025} and \cite{caputoUniversalCutoffPhenomenon2025}, where the redistribution weights are chosen randomly at each interaction.
An open line of inquiry, which we do not pursue here, is if we can obtain similar results about the convergence to the stationary measures using the dice process.

These examples illustrate the flexibility of the dice process framework, which accommodates not only genealogical dynamics involving coordinated mutation, but also a variety of particle systems governed by certain redistribution rules.
As such, dice processes provide a unified approach to modeling partial exchangeability in both population and particle systems.

\section{Exchangeable and partially exchangeable processes: Model and main results}

\subsection{Notation, exchangeability and consistency}

We start by fixing some notation, and introducing the concepts of exchangeability and consistency.
The sets of non-negative real numbers and non-negative integers are denoted by $\mathbb{R}_+ := [0, \infty)$ and $\mathbb{N}_0 := \mathbb{N} \cup \{0\}$, respectively.
For $n \in \mathbb{N}_0$, the index sets $[n] := \{1, \ldots, n\}$ and $[n]_0 := [n] \cup \{0\}$ are used.

Given $m < n \leq \infty$ and $\eta \in [d]^n$, define $R_m \eta$ as the projection of $\eta$ on $[d]^m$, meaning that $R_m \eta = (\eta_1, \ldots, \eta_m)$.
We may identify $\eta \in [d]^\infty$ with the sequence $(R_m \eta)_{m \in \mathbb{N}} \in \prod_{m \in \mathbb{N}} [d]^m$.
Giving $[d]^m$ the discrete topology and $[d]^\infty$ the topology inherited from the product $\prod_{m \in \mathbb{N}} [d]^m$, $[d]^\infty$ becomes compact and metrizable.
Additionally, we also consider the projection $\pi_m \eta := \eta_m$ for $m \leq n$ that is,  $R_m \eta=(\pi_1 \eta,...,\pi_m\eta)$.
Finally, for a generic sequence $x = (x_i)_{i \in \mathbb{N}}$ we define $x_{\sigma}$ for a finite permutation $\sigma : \mathbb{N} \to \mathbb{N}$ as $x_{\sigma} := (x_{\sigma(i)})_{i \in \mathbb{N}}$.

In this paper we will be interested in sequences $(X^{(n)})_{n \in \mathbb{N}}$ of random processes, with $X^{(n)}=(X^{(n)}(t))_{t\geq 0}$ taking values in  $[d]^n$ for each $n \in \mathbb{N}$, that are \emph{consistent, partially exchangeable,} and furthermore fulfill that \emph{$X^{(n)}$ is a Markov chain on $[d]^n$} for each $n \in \mathbb{N}$.
For fixed $n$ we will write $X_i^{(n)} = \pi_i X^{(n)}$ for the projections on the coordinates, $i\in [n]$.

We first define the concept of exchangeability and partial exchangeability, starting with exchangeability for random variables and random processes.

\begin{definition}
	A sequence $Y = (Y_i)_{i \in \mathbb{N}}$ of random variables with values in a Polish space $E$ is \emph{exchangeable} if and only if $Y_\sigma = (Y_{\sigma(i)})_{i \in \mathbb{N}}$ has the same law as $Y$ for every finite permutation $\sigma : \mathbb{N} \to \mathbb{N}$.
\end{definition}

The concept of exchangeability translates immediately for sequences of random processes with values on a Polish space $E$.

\begin{definition}\label{def:consistent_process}
	A sequence $X = \bigl((X_i(t))_{t \geq 0}\bigr)_{i \in \mathbb{N}}$ of random processes $(X_i(t))_{t\geq 0}$ with Polish state space $E$ is \emph{exchangeable} if and only if $X_\sigma = (X_{\sigma(i)})_{i \in \mathbb{N}}$ has the same distribution as $X$ for every finite permutation $\sigma : \mathbb{N} \to \mathbb{N}$.
\end{definition}

Additionally we will also consider a version of \emph{partial exchangeability}.
To this end, we will assume that $E = [d]$ and define, for every $t \geq 0$ and $i \in [d]$, the random set of indices
\begin{equation} \label{eq:defIniPartition}
	A_i(t) = \{j \in \mathbb{N} : X_j(t) = i\} \,.
\end{equation}
Note that $\mathbf{A}(t) := \{A_1(t), \ldots, A_d(t)\}$ is a partition of $\mathbb{N}$.
For a partition $\mathbf{A} := \{A_1, \ldots, A_d\}$ of $\mathbb{N}$ define $\mathrm{Perm}_{\mathbf{A}}$ as the set of finite permutations $\sigma : \mathbb{N} \to \mathbb{N}$ such that $\sigma(A_i) = A_i$ for every $i \in [d]$.
We start by stating the definition of partial exchangeability we will be considering for random variables.

\begin{definition}
	Suppose $Y = (Y_i)_{i \in \mathbb{N}}$ is a sequence of random variables with values in $[d]$.
	We say that $Y$ is partially exchangeable with respect to a partition $\mathbf{A} = \{A_1, \ldots, A_d\}$ of $\mathbb{N}$, if $Y_\sigma$ has the same law as $Y$ for every $\sigma \in \mathrm{Perm}_{\mathbf{A}}$.
\end{definition}

\begin{definition}
	A sequence $X = \bigl((X_i(t))_{t \geq 0}\bigr)_{i \in \mathbb{N}}$ of random processes $(X_i(t))_{t\geq 0}$ with values in $[d]$ is \emph{partially exchangeable}, with respect to the initial configuration, if and only if $X_{\sigma}$ has the same distribution as $X$ for every $\sigma \in \mathrm{Perm}_{\mathbf{A(0)}}$.
\end{definition}

Before passing on to the concept of consistency, a brief remark on terminology is in order.
Throughout this work, we use the term partial exchangeability for sequences of processes to indicate the following:
given an initial configuration $x \in [d]^\infty$, where $x_i$ denotes the starting state of the $i$-th process, the law of the sequence is invariant under permutations that preserve the partition of $\mathbb{N}$ induced by $x$.
This usage aligns with the notion of partial exchangeability commonly found in the Bayesian nonparametrics literature, and which also appears under the names \emph{internal exchangeability} (see Corollary 3.9 of \cite{aldousExchangeabilityRelatedTopics1985}) and \emph{separate exchangeability} (see Corollary 1.7 of \cite{kallenbergFoundationsModernProbability2021}).
The distinction lies in the setting: while the aforementioned works consider arrays of random elements $(X_{ij})_{i \in [d], j \in \mathbb{N}}$ whose laws are invariant under permutations of indices within rows, i.e. $(X_{ij})_{i \in [d], j \in \mathbb{N}}$ has the same law as $(X_{i\sigma_i(j)})_{i \in [d], j \in \mathbb{N}}$ for any choice of finite permutations $\sigma_1, \ldots, \sigma_d : \mathbb{N} \to \mathbb{N}$, we consider sequences of stochastic processes indexed by $\mathbb{N}$.

Finally we define the notion of consistency.

\begin{definition}\label{consistency}
	Consider a collection of processes $(X^{(n)})_{n \in \mathbb{N}}$ such that $X^{(n)} = (X^{(n)}(t))_{t \geq 0}$ takes values in $E^n$, where $E$ is a Polish space $E$.
	We say that $(X^{(n)})_{n \in \mathbb{N}}$ is \emph{consistent} if and only if for each $n\in \mathbb{N}$ and for every $m < n$ and every $x^{(n)} \in E^n$, $R_m X^{(n)}$ given $X^{(n)}(0) = x^{(n)}$ has the same law as $X^{(m)}$ given $X^{(m)} (0) = R_m x^{(n)}$.
\end{definition}

In the next section we will introduce the main mathematical model of our paper, the \emph{dice process}, and state the main result concerning it that relates it to a specific class of partially exchangeable sequence of Markov chains on a finite state space.
Afterwards, we will make a brief review of multitype $\Lambda$-coalescents, introduced in \cite{johnstonMultitypeLcoalescents2023} and proven to be in close relation to multitype continuous state branching processes in \cite{gonzalezcasanovaMultitypeLcoalescentsContinuous2024}.
Then we will introduce a new class of multitype coalescents, which we call \emph{multitype $\Lambda$-coalescent with multiple switching} in accordance to \cite{gonzalezcasanovaMultitypeLcoalescentsContinuous2024}, in which, additionally to the coalescence modelled by the multitype $\Lambda$-coalescent, simultaneous type switching is allowed.

\subsection{Dice processes and the class of consistent and partially exchangeable Markov chains}
\label{subsec:diceProcesses}

The aim of this section is to give a characterization of the class of partially exchangeable sequences of Markov chains $X^{(\infty)}$ such that $R_n X^{(\infty)}$ is a Markov chain on the finite state space $[d]^n$ for each $n \in \mathbb{N}$.
In order to do this, we first introduce the class of what we call \emph{dice processes}, which turn out to be the partially exchangeable sequence of processes with this property.
Write $\Delta_{d-1} := \{x \in \mathbb{R}_+^d : \abs{x} = 1\}$ for the $d$-dimensional simplex and $\Delta_{d-1}^d$ for the set of stochastic matrices of size $d \times d$.

\begin{definition}\label{def:dice_process}
	Given an array $A = (a_{ij})_{i \in [d], j \in [d] \setminus \{i\}}$ of non-negative numbers and a measure $\nu$ on $\Delta_{d-1}^d$ such that the integrability condition
	\begin{equation}
		\label{eq:nuIntegrability}
		\int_{\Delta_{d-1}^d} \sum_{i = 1}^d (1 - u_{ii}) \nu(dU) < \infty
	\end{equation}
	holds, we define the \emph{$n$-dice process} (with parameters $A$ and $\nu$) as the $[d]^n$-valued Markov process such that the transition rates are given by
	\[
		\tilde{\gamma}_{x, y}^{(n)} = a_{ij} + \int_{\Delta_{d-1}^d} u_{ij} \prod_{m \in [n] \setminus \{l\}} u_{x_m x_m} \nu(dU)
	\]
	if $x, y \in [d]^n$ are such that $x_l = i$ and $y = x + (j - i) e_l$, for some $i,j\in[d]$ and $l\in [n]$, and
	\[
		\tilde{\gamma}_{x, y}^{(n)} = \int_{\Delta_{d-1}^d} \prod_{m = 1}^n u_{x_m y_m} \nu(dU)
	\]
	for all other $x, y \in [d]^n$ such that $x \neq y$.
\end{definition}

We call this a ``dice process'' because we imagine $\nu$ as a measure sampling a ``dice'' $U$ in the set of stochastic matrices $\Delta^d_{d-1}$ which provides us with certain probabilities $u_{ij}, i,j\in[d]$.
This provides an interpretation of dice processes as a system of moving particles: Imagine $n$ labelled particles that move on the complete graph with $d$ vertices, denoted by $K_d$, jumping from vertex $i$ to vertex $j$.
The motion occurs either for only one particle at a time at rate $a_{ij}$,
or \emph{coordinated} by the collection $U=(u_{ij})_{i, j \in [d]} \in \Delta_{d-1}^d$ of  ``randomly sampled dice'',  each  with $d$ faces, where for dice $i\in[d]$ the probability of landing on face $j\in[d]$ is $u_{ij}$.
Once the dice is sampled, each particle in vertex $i \in [d]$ rolls the corresponding type $i$ dice (the $i$-th row of $U$) independently of others, and moves to the vertex $j$ indicated by the outcome of the roll.
In the dice process, $x$ corresponds to a particle configuration, with $x_l\in [d]$ being the position of particle $l\in[n]$.

\begin{lemma}\label{rem:ex_dice_infinity}
	For every $A$ and $\nu$  satisfying the conditions of Definition \ref{def:dice_process}, there exists  a $[d]^\infty$-valued process $X^{(\infty)}$ such that for each $n \in \mathbb{N}$, $X^{(n)}=R_n X^{(\infty)}$ is an $n$-dice process with parameters $A$ and $\nu$.
\end{lemma}
\begin{proof}
	Note that the integrals in Definition \ref{def:dice_process} can be written as
	\[
		\int_{\Delta_{d-1}^d} \prod_{i = 1}^d \prod_{j = 1}^d u_{ij}^{ \# \{ m \in [n] : x_m = i, y_m = j \} } \nu(dU) \,,
	\]
	and as a consequence we can rewrite the rates of the $n$-dice process as
	\begin{align}\label{aux_1}
		\tilde{\gamma}_{x, y}^{(n)} = \sum_{i = 1}^d \sum_{j \in [d] \setminus \{i\}} a_{ij} 1_{\{ \exists l \in [n] : y = x + (j - i) e_l \}} + \int_{\Delta_{d-1}^d} \prod_{i = 1}^d \prod_{j = 1}^d u_{ij}^{\#\{m \in [n] : x_m = i, y_m = j\}} \nu(dU)
	\end{align}
	for every $x, y \in [d]^n$ such that $x \neq y$.  Thus the rates $\tilde{\gamma}_{x, y}^{(n)}$ only depend on the number of particles at any sites of the graph $K_d.$
	Therefore,  for any $n\in\mathbb{N},$ if we take an $n$-dice process $X^{(n)}$ with parameters $A$ and $\nu$,  its projection $R_m X^{(n)}$ on $[d]^m$, with $m < n$, is an $m$-dice process with parameters $A$ and $\nu$.
	Thus, by Kolmogorov's extension theorem, see Theorem 8.23 in \cite{kallenbergFoundationsModernProbability2021}, there exists a $[d]^\infty$-valued process $X^{(\infty)}$ with the properties we claim.
\end{proof}

\begin{definition}\label{def:dice_infinity}
	We call the process $X^{(\infty)}$ from Lemma \ref{rem:ex_dice_infinity} the \emph{dice process} with parameters $A$ and $\nu$.  We identify $X^{(\infty)}$ with the sequence $(X^{(n)})_{n\in\mathbb{N}}$.
\end{definition}

While the definition of a dice process is rather simple, it turns out that it covers all of the class of partially exchangeable sequences of Markov chains that are consistently Markovian, which is one of our main results.

\begin{theorem} \label{thm:pExDP}
	Let $Y^{(\infty)}$ be a $[d]^\infty$-valued process such that $R_n Y^{(\infty)}$ is a $[d]^n$-valued Markov chain for each $n \in \mathbb{N}$.
	Then $Y^{(\infty)}$ is partially exchangeable if and only if it is a dice process.
	Moreover, if the latter holds, $Y^{(\infty)}$ is exchangeable if and only if $Y^{(\infty)}(0)$ is exchangeable.
\end{theorem}

Compare this theorem with Theorem 3.2 of \cite{carinciConsistentParticleSystems2021}, where Carinci et al. give sufficient conditions for an \emph{exchangeable} sequence of Markov chains to be consistent.
Their conditions are expressed through a compatible configuration process: for each $n \in \mathbb{N}$ and $t \geq 0$, let
\[
	\eta_n(t) := \sum_{i=1}^n \delta_{\pi_i Y^{(\infty)}(t)},
\]
where $\pi_i Y^{(\infty)}(t)$ denotes the position of particle $i$ at time $t$.
Consistency then requires that, for $n \geq 2$, removing a particle uniformly at random from $\eta_{n+1} := (\eta_{n+1}(t))_{t \geq 0}$ yields a process with the same distribution as $\eta_n$.
In contrast, our dice process framework operates at the level of coordinate processes, where consistency means that $R_n Y^{(\infty)}$ has the same distribution as the process obtained by removing the last coordinate of $R_{n+1} Y^{(\infty)}$.
If $Y^{(\infty)}(0)$ is exchangeable, then by Theorem \ref{thm:pExDP}, $Y^{(\infty)}$ is itself exchangable, and removing a specific coordinate would be equivalent in law to removing one uniformly at random.
Hence, a structural distinction between the two approaches lies in our setting, which accommodates partially exchangeable sequences of Markov chains.

We will prove Theorem \ref{thm:pExDP} below in Section \ref{sec:diceProcess}. The proof will depend on a characterisation of the rates in terms of what we call $(b,K)$-changes, which will be introduced in the next section.  Some implications are quite obvious from the definition, we collect them in the following remark.

\begin{remark} \label{rmk:pExchOfDice1} \label{rmk:almostExchOfDice1}
	\begin{itemize}
		\item[(1.)]  Let $X^{(\infty)}$ be a dice process with parameters $A$ and $\nu$ started from $x \in [d]^\infty$.
		      Set $A_i^x := \{n \in \mathbb{N} : x_n = i\}$ for each $i \in [d]$ and let $\sigma : \mathbb{N} \to \mathbb{N}$ be a finite permutation such that $\sigma(A_i^x) = A_i^x$ for all $i \in [d]$.
		      By \eqref{aux_1}, $X_\sigma^{(n)} = R_n X_\sigma^{(\infty)}$ is an $n$-dice process started at $R_n x$ for each $n \in \mathbb{N}$, so $X_\sigma^{(\infty)}$ is a dice process, with parameters $A$ and $\nu$, started at $x$.
		      This is, $X^{(\infty)}$ and $X_\sigma^{(\infty)}$ have the same law, so a dice process is partially exchangable.
		\item[(2.)] Let $X^{(\infty)}$ be a dice process with parameters $A$ and $\nu$ started from $x \in [d]^\infty$.
		      Let $\sigma : \mathbb{N} \to \mathbb{N}$ be a finite permutation.
		      By \eqref{aux_1}, $X_\sigma^{(n)} = R_n X_\sigma^{(\infty)}$ is an $n$-dice process started at $R_n x_\sigma$ for each $n \in \mathbb{N}$, so $X_\sigma^{(\infty)}$ is a dice process started at $x_\sigma$ with parameters $A$ and $\nu$.
		\item[(3.)]  If $X^{(\infty)}$ is a dice process, then it is a sequence of partially exchangeable Markov chains.
		      In particular, if we consider $X_1 := \eta_1 X^{(\infty)}$, then it is a Markov chain over the set $[d]$ whose rate of transition from state $i$ to state $j$, with $i \neq j$, is
		      \[
			      a_{ij} + \int_{\Delta_{d-1}^d} u_{ij} \nu(dU) \,.
		      \]
		\item[(4.)]
		      Let $X^{(\infty)}$ be a dice process with $a_{ij} = 0$ for all $i \neq j$, and let $\nu$ be a finite measure.
		      Denote by $Z_n$ the skeleton chain of $\eta_n X^{(\infty)}$.
		      Then, for each $n \in \mathbb{N}$, $Z_n$ is a discrete-time Markov chain in an independent and identically distributed random environment, sampled according to $\nu(dU) / \nu(\Delta_{d-1}^d)$.
		      Moreover, the random environment is shared by all processes $Z_n$.
		\item[(5.)]  Let $Y$ be a Markov chain on $[d]$ such that the rate of transition from state $i$ to state $j$, with $i \neq j$, is given by $\alpha_{ij}$.
		      The collection of parameters $(A, \nu)$ such that a dice process $X^{(\infty)}$ satisfies that $X_1 = \eta_1 X^{(\infty)}$ has the same distribution as $Y$ is convex.
	\end{itemize}
\end{remark}

\begin{remark}[Graphical construction]\label{rem:GC}
	It is rather straightforward to provide a graphical construction of dice processes, which also makes formula \eqref{aux_1} even more transparent. Let us consider the set $\mathbb{R}_+ \times (\mathbb{N} \times [d])$.  For each $l \in \mathbb{N}$, let $N_l$ be a Poisson point process on $\mathbb{R}_+ \times [d]^2$ with intensity $\lambda \otimes \alpha$, where
	\[
		\alpha(d\xi) = \sum_{i \in [d]} \sum_{j \in [d] \setminus \{i\}} a_{ij} \delta_{(i,j)}(d\xi) \,
	\]
	and $\lambda$ is the Lebesgue measure.
	Additionally consider another Poisson point process $N_c$ on $\mathbb{R}_+ \times [d]^{\mathbb{N} \times [d]}$ with intensity measure $\lambda \otimes L$, where
	\[
		L(d\xi) = \int_{\Delta_{d-1}^d} \nu(dU) \prod_{n = 1}^\infty \prod_{i = 1}^d \mathbf{P}_{u_i}(d\xi_{n,i}) \,.
	\]
	In the previous expression, under $\mathbf{P}_{u_i}$, $\xi$ is a random variable with values in $[d]$ such that $\mathbf{P}_{u_i}(\xi = j) = u_{ij}$ for each $i, j \in [d]$.
	The process $N_l$ corresponds to events where only particle $l$ moves,
	while $N_c$ takes care of the coordinated events.

	We assume that all of the Poisson point processes are independent.  The construction of the graph is as follows:
	\begin{itemize}
		\item If $(t, (i, j)) \in N_l$ for some $l \in \mathbb{N}$, draw an arrow from $(t, l, i)$ to $(t, l, j)$.
		\item If $(t, \xi) \in N_c$, then for each $l \in \mathbb{N}$ and $i \in [d]$ draw an arrow from $(t, l, i)$ to $(t, l, \xi_{l, i})$ if $i \neq \xi_{l, i}$.
	\end{itemize}
	We start with a initial configuration $\eta(0) = (\eta_l(0))_{l \in \mathbb{N}}$, where $\eta_l(0) \in [d]$ for each $l$ means that there is a particle in $(0, l, \eta_l(0))$ and no particles in $(0, l, i)$ for $i \in [d] \setminus \{\eta_l(0)\}$.
	From this starting configuration we let the particles evolve in time following the arrows that were drawn over the graph.
	For each time $t \geq 0$, $\eta_l(t)$ represents the position of the $l$-th particle at time $t$.
	An example is presented in Figure \ref{fig:graph}, where the graph corresponding to one random walk with $d=6$ states is given.
	In the same example it becomes clear why reversing the process does not yield a Markov chain: Due to the coordinated jumps, there is in general no unique backward arrow at a large event.

	\begin{figure}[htb]
		\centering
		% \begin{tikzpicture}
		% 	\foreach \jumptime in {5,8} {
		% 			\draw[thick,dotted] (\jumptime,0) -- (\jumptime,6);
		% 		}
		% 	\draw[ultra thick,-latex] (0, 0) -- (10.5, 0) node[below]{$t$};
		% 	\foreach \a in {1,2,...,6} {
		% 			\draw[ultra thick] (0,\a) node [left] {$\a$} -- (10,\a);
		% 		}
		% 	\foreach \a/\b/\time in {2/3/1,3/6/3.5} {
		% 			\draw[ultra thick,-latex] (\time,\a) to [out=150,in=260] (\time,\b);
		% 		}
		% 	\foreach \a/\b/\time in {5/1/6.7,6/5/1.7,3/1/9.5} {
		% 			\draw[ultra thick,-latex] (\time,\a) to [out=210,in=100] (\time,\b);
		% 		}
		% 	\foreach \i/\j/\time in {1/3/5,2/5/5,3/4/5,1/3/8,3/4/8} {
		% 			\draw[ultra thick,-latex] (\time,\i) to [out=150,in=260] (\time,\j);
		% 		}
		% 	\foreach \i/\j/\time in {6/5/5,4/1/8,5/1/8,6/5/8} {
		% 			\draw[ultra thick,-latex] (\time,\i) to [out=210,in=100] (\time,\j);
		% 		}
		% 	\draw[red,ultra thick]
		% 	(0,2) -- (1,2)
		% 	(1,2) edge [out=150,in=260,-latex] (1,3)
		% 	(1,3) -- (3.5,3)
		% 	(3.5,3) edge [out=150,in=260,-latex] (3.5,6)
		% 	(3.5,6) -- (5,6)
		% 	(5,6) edge [out=210,in=100,-latex] (5,5)
		% 	(5,5) -- (6.7,5)
		% 	(6.7,5) edge [out=210,in=100,-latex] (6.7,1)
		% 	(6.7,1) -- (8,1)
		% 	(8,1) edge [out=150,in=260,-latex] (8,3)
		% 	(8,3) -- (9.5,3)
		% 	(9.5,3) edge [out=210,in=100,-latex] (9.5,1)
		% 	(9.5,1) -- (10,1);
		% \end{tikzpicture}
		\includegraphics{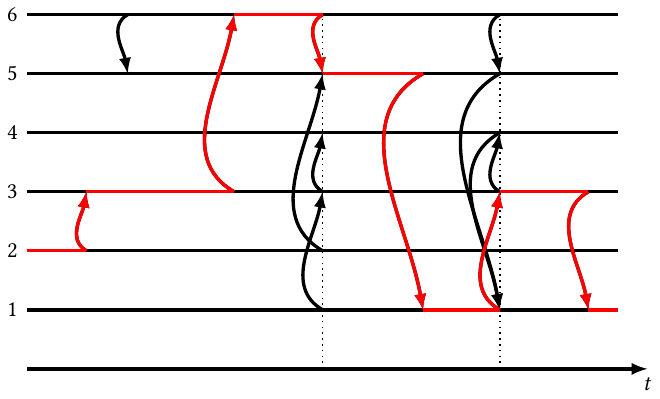}
		\caption{Example of a graph for an arbitrary particle. The dotted line correspond to times where coordination occurs. The red line corresponds to the path of a particle that started at position $2$.}
		\label{fig:graph}
	\end{figure}

\end{remark}

\subsection{Consistency condition for the rates of dice processes}\label{subsect:particle_interpretation}

We will now have a closer look at the rates of a dice process, and rewrite them in a way that will be suitable to state the condition under which they are consistent.

We will need some more notation.
For $n \in \mathbb{N}_0$, the set $\mathcal{S}_{d-1}(n) := n \Delta_{d-1} \cap \mathbb{N}_0^d = \{k \in \mathbb{N}_0^d : k_1+....+k_d = n\}$ consists of all vectors in $\mathbb{N}_0^d$ whose components sum to $n$.
For $n \in \mathbb{N}_0^d$, the set $\mathcal{S}_{d-1,d}(n) := \mathcal{S}_{d-1}(n_1) \times \cdots \times \mathcal{S}_{d-1}(n_d)$ is the Cartesian product of these sets.
A typical element of the latter is written as a matrix $K = (k_{ij})_{i,j=1}^d$, where $k_i := (k_{i1}, \ldots, k_{id}) \in \mathcal{S}_{d-1}(n_i)$ for $i \in [d]$.

The diagonal matrix $\diag(n)$ for $n \in \mathbb{N}_0^d$ is defined by $k_{ij} = n_i \delta_{ij}$ for all $i, j \in [d]$.
For $x \in \mathbb{R}^d$ and $n \in \mathbb{N}_0^d$, we define $x^n := \prod_{i=1}^d x_i^{n_i}$.
The multinomial coefficient for $n \in \mathbb{N}_0$ and $k \in \mathcal{S}_{d-1}(n)$ is given by $\binom{n}{k} := \frac{n!}{k_1! \cdots k_d!}$.

Let us now look again at the rates of the dice process. Let a particle configuration for the dice process be given by a vector $x \in [d]^n$, where $x_l\in[d]$ denotes the position of the $l$-th particle.
Suppose that the $n$ particles move from a configuration $x \in [d]^n$ to a configuration $y \in [d]^n$ at a rate $\tilde{\gamma}_{x, y}$ that only depends on the number of particles that move from vertex $i$ to vertex $j$ for each $i, j \in [d]$.
This is, if $x', y' \in [d]^n$ are other vectors such that
\begin{equation} \label{eq:equalChanges}
	\#\{l \in [n] : x_l = i, y_l = j\}
	= \#\{l \in [n] : x_l' = i, y_l' = j\} \quad\text{ for all } i, j \in [d] \,,
\end{equation}
then
\begin{equation} \label{eq:pexRates}
	\tilde{\gamma}_{x, y} = \tilde{\gamma}_{x', y'} \,.
\end{equation}
Consider $b \in \mathcal{S}_{d-1}(n)$ and $K \in \mathcal{S}_{d-1,d}(b) \setminus \{\diag (b)\}$.
In the context of the particles, we can understand $b_i$ as the number of particles that are in vertex $i$ right before a jump, and $k_{ij}$ as the number of particles that moved from vertex $i$ to vertex $j$.
For a given $b$ and $K$, if $x, y \in [d]^n$ are such that
\begin{equation} \label{eq:bkDef}
	b_i = \#\{l \in [n]: x_l = i\}
	\quad\text{and}\quad
	k_{ij} = \#\{l \in [n] : x_l = i, y_l = j\} \,,
\end{equation}
let us put
\[
	\gamma_{b, K} = \tilde{\gamma}_{x, y} \,.
\]
By \eqref{eq:equalChanges},  $\gamma_{b, K}$ is well defined.
Therefore, we say that \emph{a $(b, K)$-change occurred} in the particle system if and only if it moved from a state $x \in [d]^n$ to a state $y \in [d]^n$ such that \eqref{eq:bkDef} holds.
Observe that for each choice of $b$ and $K$, there are
\[
	\frac{n!}{\prod_{i = 1}^d \prod_{j = 1}^d k_{ij}!}
\]
different choices of $x$ and $y$ that satisfy \eqref{eq:bkDef},
while for each $x$ and a given $K$, there are
\[
	\prod_{i = 1}^d \frac{b_i!}{\prod_{j = 1}^d k_{ij}!}
\]
different choices of $y$ such that a move from $x$ to $y$ is a $(b, K)$-change.
By the previous description, if we define $X_i^{(n)}(t)$ as the location of particle $i$ at time $t$, then $X^{(n)} := (X_1^{(n)}, \ldots, X_n^{(n)})$ is a Markov chain on $[d]^n$.

\begin{remark} \label{rem:bkChangesDice}
	As we saw in the proof of Lemma  \ref{rem:ex_dice_infinity}, the rates of a dice process agree for configurations such that \eqref{eq:pexRates} is satisfied.
	Therefore,  for a dice process, $R_n X^{(\infty)}$ only has $(b, K)$-changes for each $n \in \mathbb{N}$.
\end{remark}

We now turn to show that if a process $X^{(\infty)}$ is such that $R_n X^{(\infty)}$ is a Markov chain on $[d]^n$ that only has $(b, K)$-changes, then it must be a dice process.
To proceed we will now suppose that for each $n \in \mathbb{N}$, $X^{(n)} = (X_i^{(n)})_{i \in [n]}$ is a $[d]^n$-valued Markov chain with only $(b, K)$-changes determined by collection of rates
\[
	\gamma_{b, K} \quad\text{with }
	b \in \mathcal{S}_{d-1}(n) \text{ and }
	K \in \mathcal{S}_{d-1,d}(b) \setminus \{\diag(b)\} \,.
\]

\begin{definition} In the setup of this subsection, we say that the collection of rates
	\begin{equation} \label{eq:gammaRates}
		\gamma = \bigl( \gamma_{b, K} : b \in \mathbb{N}_0^d \setminus \{0\}, K \in \mathcal{S}_{d-1,d}(b) \setminus \{\diag(b)\} \bigr)
	\end{equation}
	is \emph{consistent} if and only if the sequence of processes $(X^{(n)})_{n \in \mathbb{N}}$ is consistent in the sense of Definition  \ref{consistency}.
\end{definition}

The following lemma characterizes, in terms of an algebraic relation, the transition rates of a consistent sequence of Markov chains whose rates are given by the collection $\gamma$ defined in \eqref{eq:gammaRates}.
This result mimics Lemma 18 in \cite{pitmanCoalescentsMultipleCollisions1999}, and it is an immediate consequence of Theorem 4 of \cite{burkeMarkovianFunctionMarkov1958}.  Its proof is provided in Section \ref{sec:diceProcess}.
Let $E=(E_{i,j})_{i,j\in [d]^2}$ be the elementary matrix with a $1$ at position $(i, j) \in [d]^2$ and zeros elsewhere, that is, $E_{ij} := (\delta_{il} \delta_{j\ell})_{l,\ell\in [d]^2}$.

\begin{lemma} \label{lem:consistencyArray}
	Let $\gamma = (\gamma_{b, K} : b \in \mathbb{N}_0^d \setminus \{0\}, K \in \mathcal{S}_{d-1,d}(b) \setminus \{\diag(b)\})$ be an array of nonnegative real numbers.
	The array of rates $\gamma$ is consistent if and only if
	\begin{equation} \label{eq:consistencyEq}
		\gamma_{b, K} = \sum_{l = 1}^d \gamma_{b + e_j, K + E_{jl}}
		\quad\text{for all } j \in [d]\,.
	\end{equation}
\end{lemma}

The following lemma gives an explicit representation of the rates $\gamma_{b, K}$ whenever the array $\gamma$ is consistent. This representation transforms the jumping particle/dice process point of view into the rates of a $(b,K)$ transition of particle configurations, which only depend on the number of particles with transit from one vertex to another.

\begin{lemma} \label{lem:bkRatesExplicit}
	Let $\gamma$ be as in Lemma \ref{lem:consistencyArray}.
	Then $\gamma$ is consistent if and only if there exist a measure $\nu$ on $\Delta_{d-1}^d$ and an array of constants $(a_{ij} : i \in [d], j \in [d] \setminus \{i\})$ such that
	\begin{equation} \label{eq:nuConstraint}
		\int_{\Delta_{d-1}^d} \sum_{i = 1}^d (1 - u_{ii}) \nu(dU) < \infty \,,
	\end{equation}
	and
	\begin{equation} \label{eq:transitionCharact}
		\gamma_{b, K} = \int_{\Delta_{d - 1}^d} \prod_{i = 1}^d \prod_{j = 1}^d u_{ij}^{k_{ij}} \nu(dU) + \sum_{i = 1}^d \sum_{j \in [d] \setminus \{i\}} a_{ij} 1_{\{K = \diag(b) - E_{ii} + E_{ij}\}} \,.
	\end{equation}
\end{lemma}

With this lemma, whose proof is also provided in section \ref{sec:diceProcess}, we are ready to prove the main result of this section.

\begin{proposition} \label{th:cerwCharact}
	Let $\gamma$ be a collection of nonnegative real numbers as in Lemma \ref{lem:consistencyArray} and $x^{(\infty)} \in [d]^\infty$ given.
	Then there exists a $[d]^\infty$-valued process $X^{(\infty)}$ with $X^{(\infty)}(0) = x$ and such that for each $n$, $R_n X^{(\infty)}$ is a Markov chain on $[d]^n$ that only observes $(b, K)$-changes and whose rates are determined by the subarray
	\begin{equation} \label{eq:nbkRates}
		(\gamma_{b, K} : b \in \mathcal{S}_{d-1}(n), K \in \mathcal{S}_{d-1,d}(b) \setminus \{\diag(b)\})
	\end{equation}
	if and only if there exist a measure $\nu$ on $\Delta_{d-1}^d$ and an array of nonnegative constants $(a_{ij} : i \in [d], j \in [d] \setminus \{i\})$ such that \eqref{eq:nuConstraint} and \eqref{eq:transitionCharact} hold,
	in which case $X^{(\infty)}$ is a dice process.
\end{proposition}

\begin{proof}
	Assume first that $\nu$ and $A = (a_{ij} : i \in [d], j \in [d] \setminus \{j\})$ exist and satisfy \eqref{eq:nuConstraint} and \eqref{eq:transitionCharact}.
	Let $X^{(\infty)}$ be a dice process with parameters $A$ and $\nu$.
	By Remark \ref{rem:bkChangesDice}, $R_n X^{(\infty)}$ only has $(b, K)$-changes at rate $\gamma_{b, K}$.

	Conversely, assume that $X^{(\infty)}$ exists and so $\gamma$ is consistent.
	By Lemmas \ref{lem:consistencyArray} and \ref{lem:bkRatesExplicit} there exist an array $A$ of nonnegative constants and a measure $\nu$ on $\Delta_{d-1}^d$ such that \eqref{eq:nuConstraint} and \eqref{eq:transitionCharact} hold.
	It is clear that then $X^{(\infty)}$ is a dice process, cf. \eqref{aux_1}.
\end{proof}

Before moving on to the de Finetti measure, we compare Proposition \ref{th:cerwCharact} with Theorem 3.3 in \cite{carinciConsistentParticleSystems2021}.
In that theorem, the authors characterize the rates of consistent configuration processes in which only a single particle jumps at each time.
In contrast, our framework characterizes the rates of consistent coordinate processes, where multiple particles may jump simultaneously.

\subsection{De Finetti measure and moment duality}
In this section we consider an exchangeable dice process $X^{(\infty)}$.
By the last assertion of Theorem \ref{thm:pExDP}, this means that in this section we ask that $X^{(\infty)}(0)$ is exchangeable.
By definition, this implies that for each $t \geq 0$, $X^{(\infty)}(t) = (X_n^{(\infty)}(t))_{n \in \mathbb{N}}$ is an exchangable sequence with values in $[d]$.
Therefore, by de Finetti's representation theorem, see Theorem 3.1 in \cite{aldousExchangeabilityRelatedTopics1985}, there exists a probability measure $\gamma_t$ on $\Delta_{d-1}$, seen as the set of probability measures on $[d]$, such that
\begin{equation} \label{eq:introdeFinetti}
	\mathbb{P}\Bigl( \bigcap_{n = 1}^N \{X_n^{(\infty)}(t) = i_n\} \Bigr) = \int_{\Delta_{d-1}} \prod_{j = 1}^d u_j^{\#\{n \in [N] : i_n = j\}} \gamma_t(du)
\end{equation}
holds for every $N \in \mathbb{N}$ and each $i_1, \ldots, i_N \in [d]$.
Let $R(t)$ be a random vector on $\Delta_{d-1}$ with distribution $\gamma_t$, and let us rewrite \eqref{eq:introdeFinetti} as
\[
	\mathbb{P}\Bigl( \bigcap_{n = 1}^N \{X_n^{(\infty)}(t) = i_n\} \Bigr) = \mathbb{E}\Bigl[ \prod_{j = 1}^d R_j(t)^{\#\{n \in [N] : i_n = j\}} \Bigr] \,.
\]
Therefore, what \eqref{eq:introdeFinetti} tells us is that, conditional on $R(t)$,  $(X_{n}^{(\infty)}(t))_{n \in \mathbb{N}}$ is a sequence of independent and identically distributed random variables such that $R_i(t)$ is the probability that, for any fixed $n \in \mathbb{N}$, $X_n^{(\infty)}(t)$ is equal to $i$.

We remark that this is equivalent to studying the frequency process, defined as the asymptotic proportion of Markov chains that are in a state $i$ for each $i \in [d]$.
Indeed, by exchangeability at a fixed time $t \geq 0$, letting $R(t)$ be as in the previous paragraph, then
\[
	R(t) = \lim_{n \to \infty} \frac{1}{n} \sum_{j = 1}^n \delta_{X_j^{(\infty)}(t)}
\]
almost surely in the weak topology of probability measures on $[d]$ by Theorem 3.1 and Lemma 2.15 of \cite{aldousExchangeabilityRelatedTopics1985}.
This is readily seen to be equivalent to the almost sure convergence
\begin{align}\label{def_r}
	R_i(t) = \lim_{n \to \infty} \frac{\#\{j \in [n] : X_j^{(\infty)}(t) = i\}}{n}
\end{align}
for all $i \in [d]$, so $R(t) := (R_1(t), \ldots, R_d(t)) \in \Delta_{d-1}$ can also be understood as the frequency process at time $t$.
We call $R(t)$ the \emph{de Finetti measure} of $X^{(\infty)}(t)$.
We are interested in the random process $R := (R(t))_{t \geq 0}$ with values in $\Delta_{d-1}$ which, abusing notation, we also deem as \emph{de Finetti measure}. In the next result we extend the convergence in \eqref{def_r} to the space $D([0, T], \Delta_{d-1})$ equipped with the Skorokhod J1 topology, and characterize the process $R$ as the unique solution to a SDE. We defer its proof to Appendix \ref{Prop_2.2}.

\begin{proposition}\label{prop:convergence_deFinetti}
	The sequence of processes $(R^{(n)})_{n \geq 1}$, with
	\[
		R_i^{(n)}(t) := \frac{\#\{l \in [n] : X_l^{(\infty)}(t) = i\}}{n}
		\quad\text{for $i \in [d]$ and $t \geq 0$,}
	\]
	converges weakly in the space $D([0, T], \Delta_{d-1})$ equipped with the Skorokhod J1 topology, to the process $R$ defined in \eqref{def_r}. Additionally $R$ is the unique strong solution to
	\begin{equation} \label{eq:deFinettiSDE}
		dR_i(t) = \sum_{j \in [d] \setminus \{i\}} \bigl( a_{ji} R_j(t) - a_{ij} R_i(t) \bigr) dt
		+ \int_{\Delta_{d-1}^d} \sum_{j \in [d] \setminus \{i\}} \bigl( u_{ji} R_{j}(t-) - u_{ij} R_i(t-) \bigr) N(dt, dU) \,,
	\end{equation}
	where $N$ is a Poisson random measure on $\mathbb{R}_+ \times \Delta_{d-1}^d$ whose intensity measure is $\lambda \otimes \nu$.
\end{proposition}

The corresponding generator $\mathcal{A}$ acts formally on functions $f \in C^2(\Delta_{d-1})$ by
\begin{equation} \label{eq:generatorFinetti}
	\begin{split}
		\mathcal{A} f(r)
		 & = \sum_{i \in [d]} \sum_{j \in [d] \setminus \{i\}} a_{ij} r_i (\partial_j - \partial_i) f(r)
		+ \int_{\Delta_{d-1}^d} \bigl( f(U^T r) - f(r) \bigr) \nu(dU)                                      \\
		 & = \sum_{i \in [d]} \sum_{j \in [d] \setminus \{i\}} (a_{ji} r_j - a_{ij} r_i) \partial_{i} f(r)
		+ \int_{\Delta_{d-1}^d} \bigl( f(U^T r) - f(r) \bigr) \nu(dU) \,.
	\end{split}
\end{equation}

As a limit of frequency processes, it is natural to expect $R$ to satisfy a moment duality.
This is, that there exists a Markov process $N = (N_t)_{t \geq 0}$ with values in $\mathbb{N}_0^d$ such that
\[
	\mathbb{E}_r\Bigl[ \prod_{i = 1}^d R_i(t)^{b_i} \Bigr] = \mathbb{E}_{b}\Bigl[ \prod_{i = 1}^d r_i^{N_i(t)} \Bigr] \,,
\]
holds for any $r \in \Delta_{d - 1}$ and $b \in \mathbb{N}_0^d$.
In the previous display, $\mathbb{E}_r$ means that we take the expectation with respect to the de Finetti measure started at $r \in \Delta_{d-1}^d$ and $\mathbb{E}_b$ means that we take the expectation with respect to $N$ started at $b \in \mathbb{N}_0^d$.
The existence of a moment dual holds under an assumption that resembles a reversibility condition, as is stated in the following result.
Moreover, under this condition, the dual process is the counting process of another dice process.

\begin{proposition}\label{prop:duality}
	Assume that $\sum_{j \in [d] \setminus \{i\}} (a_{ij} - a_{ji}) = 0$ for all $i \in [d]$ and that $\nu$ has support on doubly stochastic matrices.
	Then $R$ has a moment dual $N$, which is a Markov process on $\mathbb{N}_0^d$ with transition rates
	\[
		q_{bb'} = \begin{dcases*}
			b_i a_{ji}                                               & if $b' = b + e_j - e_i$,                                              \\
			\prod_{i = 1}^d \binom{n_i}{k_i} \widehat{\gamma}_{b, K} & if $b_i' = \sum_{j = 1}^d k_{ji}$ for $K \in \mathcal{S}_{d-1,d}(b)$, \\
			0                                                        & otherwise.
		\end{dcases*}
	\]
	In the previous expression,
	\[
		\widehat{\gamma}_{n, K}
		:= \int_{\Delta_{d-1}^d} \prod_{i = 1}^d \prod_{j = 1}^d u_{ij}^{k_{ji}} \mathbf{T}^{\mathrm{transp}}\nu(dU)
	\]
	with $\mathbf{T}^{\mathrm{transp}}\nu(dU)$ being the pushforward measure of $\nu$ under the mapping $U \mapsto U^T$.
\end{proposition}

\subsection{Multitype \texorpdfstring{$\mathbf{\Lambda}$}{L}-coalescents: Characterisation in terms of dice processes}

The class of $\Lambda$-coalescents was introduced, independently,  by Donnelly and Kurtz \cite{donnellyParticleRepresentationsMeasureValued1999},
Pitman \cite{pitmanCoalescentsMultipleCollisions1999} and Sagitov \cite{Sagitov1999}.
It was shown to be characterized as the class of Markov processes with values in the partitions of $\mathbb{N}$ that satisfy:
\begin{itemize}
	\item The transitions consist of mergers of at least two blocks.
	\item \emph{Exchangeability}: The distribution of the process is invariant under permutation of labels.
	\item \emph{Consistency}: The restriction of the process to partitions of $[n]$ is Markovian for any $n$ and consistent in the sense of Definition \ref{def:consistent_process}.
	\item \emph{Asynchronicity}: No simultaneous occurrence of mergers.
\end{itemize}
Theorem 1 in \cite{pitmanCoalescentsMultipleCollisions1999} states that such a process is characterized by a finite measure $\Lambda$ on $[0, 1]$.
Indeed, given $\Lambda$, the rate at which $k$ blocks merge whenever there are $b$ blocks is given by
\[
	\lambda_{b, k} = \Lambda(\{0\}) 1_{\{k = 2\}} + \int_{(0, 1]} u^k (1 - u)^{b - k} \frac{\Lambda(du)}{u^2} \,.
\]

Recently, the extension of $\Lambda$-coalescents to multitype setup has garnered attention.
Johnston, Kyprianou and Rogers introduced the \emph{multitype-$\Lambda$-coalescent} (M-$\Lambda$-coalescent for short) in \cite{johnstonMultitypeLcoalescents2023} as processes with values in $d$-typed partitions of $\mathbb{N}$, meaning that each block of the partition is assigned one out of $d$ types.
The possible transitions consist of mergers of multiple blocks of possibly different types into one block of one type, and type switches of single blocks.
This extension was shown to be natural in \cite{gonzalezcasanovaMultitypeLcoalescentsContinuous2024} as the multitype coalescents proposed by Johnston et al. are in close relationship to multitype continuous state branching processes.
Specifically, \cite{gonzalezcasanovaMultitypeLcoalescentsContinuous2024} proved that M-$\Lambda$-coalescents arise as images of multitype continuous state branching process through what they call the \emph{Gillespie} transformation.

To describe M-$\Lambda$-coalescents in a more formal way, let us recall from \cite{gonzalezcasanovaMultitypeLcoalescentsContinuous2024} the space of $d$-typed partitions of $\mathbb{N}$.

\begin{definition}
	We say that $\bm{\pi} = \{(\pi_i, \mathfrak{t}_i) : i \in \mathbb{N}\}$ is a $d$-type partition of $\mathbb{N}$ if $\pi = \{\pi_i : i \in \mathbb{N}\}$ is a partition of $\mathbb{N}$ and for each $i \in \mathbb{N}$, $\mathfrak{t}_i \in [d]_0$ is the type of block $\pi_i$, where $\mathfrak{t}_i = 0$ if and only if $\pi_i = \varnothing$.
	We assume that the elements of $\pi$ are arranged by their minimal elements; this is $\min \pi_i \leq \min \pi_j$ for $i < j$, with the convention $\min \varnothing = \infty$, where there is strict inequality if $\pi_i \neq \varnothing$.
	Let $\mathbf{P}_d$ be the collection of all $d$-type partitions of $\mathbb{N}$.
	For $m \in \mathbb{N}$ and $\bm{\pi} \in \mathbf{P}_d$ we put $\bm{\pi}\vert_{[m]} := \{(\pi_i \vert_{[m]}, \mathfrak{t}_i 1_{\{\pi_i\vert_{[m]} \neq \varnothing\}}) : i \in \mathbb{N}\}$ for the restriction of $\bm{\pi}$ to $[m]$, where $\pi_i\vert_{[m]} := \pi_i \cap [m]$.
\end{definition}

Consider a given typed partition $\bm{\pi} \in \mathbf{P}_d$.
Let $\#\pi$ denote the number of non-empty blocks of $\bm{\pi}$ and let $J \subset [\#\pi]$ be given, where $[\infty]$ is understood as $\mathbb{N}$.
For $i \in [d]$ let us define $\coal_J^i(\bm{\pi})$ as the typed partition obtained from $\bm{\pi}$ by merging the blocks $\pi_j$ indexed by $j \in J$, and assigning this new block the type $i$.
Using this terminology, an M-$\Lambda$-coalescent has a formal infinitesimal generator $\mathcal{L}_{\mathrm{coal}}$ that acts on functions $f \in C(\mathbf{P}_d)$ as
\begin{equation}\label{eq:Lcoal}
	\mathcal{L}_{\mathrm{coal}} f(\bm{\pi}) = \sum_{i = 1}^d \sum_{J \subset [\#\pi]} \Bigl( f\bigl(\coal_J^i(\bm{\pi})\bigr) - f(\bm{\pi}) \Bigr) \lambda_{\bm{\pi}, J, i} \,,
\end{equation}
where, according to \cite{johnstonMultitypeLcoalescents2023}, for each $i \in [d]$,
\begin{equation} \label{eq:jkrRates}
	\begin{split}
		\lambda_{\bm{\pi}, J, i}
		 & = \sum_{k \in [d] \setminus \{i\}} \rho_{ik} 1_{\{\mathfrak{t}_j = k, j \in J\}} 1_{\{\abs{J} = 1\}}
		+ \rho_{ii} 1_{\{\mathfrak{t}_j = i, j \in J\}} 1_{\{\abs{J} = 2\}}                                                                                                            \\
		 & \quad + \int_{[0, 1]^d} \prod_{k = 1}^d u_k^{\abs{ \{j \in J : \mathfrak{t}_j = k\} }} (1 - u_k)^{\abs{ \{ j \in [\#\pi] \setminus J : \mathfrak{t}_j = k \} }} Q_i(du) \,,
	\end{split}
\end{equation}
where $\rho_{ik} \geq 0$ for every $k \in [d]$ and $Q_i$ is a measure on $[0, 1]^d$ without atom at zero such that
\[
	\int_{[0, 1]^d} \sum_{k = 1}^d u_{k}^{1 + \delta_{ik}} Q_i(du) < \infty \,.
\]
In \eqref{eq:jkrRates}, $\rho_{ik}$ represents the rate at which a block switches its type from $k$ to $i$ independently of the others, $\rho_{ii}$ the rate at which two blocks of type $i$ merge into a block of type $i$, and the integral term gives the rate at which blocks, of possibly different types, merge to form a block of type $i$.
Note that by construction in the M-$\Lambda$-coalescent, multiple type transitions of blocks cannot occur simultaneously.
This is in line with the idea that for $\Lambda$-coalescents there is asynchronity.  However,  if one interprets asynchronity as the absence of simultaneous (multiple) \emph{mergers},  a priori simultaneous type switches are not prohibited.
Considering models like the seedbank coalescent with simultaneous switches \cite{Blathetal2020} makes clear that this is also a valid generalisation of the definition of $\Lambda$-coalescents to the multitype set up, in particular if one has Pitman's characterisation in mind.
To be very precise we now ask that in our coalescent processes \emph{no other event may take place when a merger occurs}.
Thus we explicitly allow for simultaneous type switches (or simultaneous mutations, in a different interpretation).
Hence we study multitype coalescent processes in the sense of Pitman as presented at the beginning of this section.

We call the mathematical object \emph{multitype $\Lambda$-coalescent with multiple switching} (M-$\Lambda$-MS-coalescent for short), and formally define it as follows.

\begin{definition} \label{def:PMLcoal}
	A multitype $\Lambda$-coalescent with multiple switching $\mathbf{\Pi} = (\mathbf{\Pi}(t))_{t \geq 0}$ is a multitype coalescent process with values in $\mathbf{P}_d$ that is partially exchangeable, consistent, Markovian, and where no other event may take place at the same time that a merger occurs.
\end{definition}

For a multitype coalescent process $\mathbf{\Pi}$ we let $X_i(t)\in[d]_0$ denote the type of the $i$-th block of $\mathbf{\Pi}(t)$ at time $t \geq 0$.
Before proceeding, let us clarify the notion of partial exchangeability used in Definition \ref{def:PMLcoal}.
For a finite permutation $\sigma : \mathbb{N} \to \mathbb{N}$ we define $\mathbf{\Pi}_\sigma := (\mathbf{\Pi}_\sigma(t))_{t \geq 0}$ by considering that $\sigma(i)$ and $\sigma(j)$ are in the same block of type $l$ of $\mathbf{\Pi}_\sigma(t)$ if and only $i$ and $j$ are in the same block of type $l$ of $\mathbf{\Pi}(t)$.
Let $\mathbf{A}(0)$ be the partition of $\mathbb{N}$ induced by $(X_i(0))_{i \in \mathbb{N}}$, as in \eqref{eq:defIniPartition}.
Then the notion of partial exchangeability is with respect to $\mathbf{A}(0)$; this is, $\mathbf{\Pi}_{\sigma}$ has the same law as $\mathbf{\Pi}$ for every finite permutation $\sigma$ such that $\sigma(A) = A$ for all $A \in \mathbf{A}(0)$.

Before turning to the main result of this section, the characterization of M-$\Lambda$-MS-coalescents, let us use the dice process to formally model type switching of blocks in a process that takes values in $\mathbf{P}_d$.
Take $\bm{\pi} \in \mathbf{P}_d$ fixed.
For $\tilde{t} \in [d]^{\#\pi}$ define $\muta(\bm{\pi}, \tilde{\mathfrak{t}})$ as the typed partition obtained from $\bm{\pi}$ by replacing $\mathfrak{t}_j$ for $\tilde{\mathfrak{t}}_j$ for each $j \in [\#\pi]$.
We define the formal generator $\mathcal{L}_{\mathrm{dice}}$ acting on functions $f \in C(\mathbf{P}_d)$ as
\begin{equation} \label{eq:Ldice}
	\mathcal{L}_{\mathrm{dice}} f(\bm{\pi}) = \sum_{\tilde{\mathfrak{t}} \in [d]^{\#\pi}} \Bigl( f\bigl(\muta(\bm{\pi}, \tilde{\mathfrak{t}})\bigr) - f(\bm{\pi}) \Bigr) \gamma_{\bm{\pi}, \tilde{\mathfrak{t}}} \,.
\end{equation}

In the previous expression,
\begin{align*}
	\gamma_{\bm{\pi}, \tilde{\mathfrak{t}}}
	 & = \sum_{i = 1}^d \sum_{k \in [d] \setminus \{i\}} a_{ik}  \mathcal{I}_{i,k}(\mathfrak{t}, \tilde{\mathfrak{t}})
	+ \int_{\Delta_{d-1}^d} \prod_{i = 1}^d \prod_{k = 1}^d u_{ik}^{ \abs{ \{j \in [\#\pi] : \mathfrak{t}_j = i, \tilde{\mathfrak{t}}_j = k\} } } \nu(dU) \,,
\end{align*}
for an array $A = (a_{ik}, i \in [d], k \in [d] \setminus \{i\})$ of non-negative numbers and a measure $\nu$ on $\Delta_{d-1}^d$ taken as in section \ref{subsec:diceProcesses}, and $\mathcal{I}_{i,k}(\mathfrak{t}, \mathfrak{\tilde{t}})$ is equal to $1$ if there is a unique index $j \in [\#\pi]$ such that $\mathfrak{t}_j \neq \tilde{\mathfrak{t}}_j$, and for such index $\mathfrak{t}_j = i$ and $\tilde{\mathfrak{t}}_j = k$, and equal to $0$ otherwise, hence we are in the setting of dice processes.
Consider now the the formal generator $\mathcal{L}$ defined by
\begin{equation} \label{eq:PMLgen}
	\mathcal{L} = \mathcal{L}_{\mathrm{coal}} + \mathcal{L}_{\mathrm{dice}} \,.
\end{equation}
From now on, will assume that in \eqref{eq:jkrRates} $\rho_{ik} = 0$ for $i \neq k$, as these rates are interpreted as a type transition rates and we are modeling type transitions using the Dice process.
This assumption entails that $\mathcal{L}_{\mathrm{coal}}$ only encodes the coalescence mechanism of a process with values in $\mathbf{P}_d$.

While only formal in nature, this construction can be realised via the graphical construction of Remark \ref{rem:GC} using independent Poisson point processes for the coalescent part and for the type transition part.
Hence we know that M-$\Lambda$-MS-coalescents exist.

From now on, when we say that a coalescent mechanism and a type transition mechanism are independent, we mean that their underlying Poisson point processes are independent.
The following result states that all of the M-$\Lambda$-MS-coalescents can be written as the superposition of an independent M-$\Lambda$-coalescent and a dice process.

\begin{theorem} \label{thm:pmLambda}
	A process $\bm{\Pi}$ with values in $\mathbf{P}_d$ is a multitype $\Lambda$-coalescent with multiple switching if and only if it has a coalescent mechanism given by a multitype-$\Lambda$-coalescent and a type transition mechanism given by a dice process, which is independent of the coalescent mechanism.
	This is, a coalescent process $\bm{\Pi}$ is a multitype $\Lambda$-coalescent with multiple switching if and only if its generator $\mathcal{L}$ satisfies decomposition \eqref{eq:PMLgen} for some generators $\mathcal{L}_{\mathrm{coal}}$ and $\mathcal{L}_{\mathrm{dice}}$ of the form given in \eqref{eq:Lcoal} and \eqref{eq:Ldice}, respectively.
\end{theorem}

Contrast this result with the works of Dai Pra et al. and Flamm and Möhle, which explore different classes of multitype coalescent processes.
In the recent work of \cite{daipraMultitypeXcoalescentsStructured2025}, a class of multitype $\Xi$-coalescents is introduced that allows for simultaneous coalescent events and type switching, under the condition that all blocks participating in such events are of the same type.
The rates are specified in a probabilistic way, rather than through an explicit integral representation.
This construction further illustrates the need for integral characterizations of multitype coalescent processes, particularly in settings involving both mergers and mutation.

Meanwhile, in Proposition 2 and Theorem 3 of \cite{flammGenealogyMultitypeCannings2025}, integral representations are provided for the rates of a specific subclass of multitype coalescent processes. Specifically, the class they consider allows for multiple coalescent events to occur simultaneously, with blocks of different types potentially coalescing at the same time, with the caveat that blocks of type $k$ must merge into a block of type $k$.
That is, their framework allows simultaneous mergers but does not account for type switching (or mutation) during coalescent events.

In contrast, our work provides an if and only if characterization of multitype coalescent processes that are consistent, partially exchangeable, and Markovian, and in which either:
\begin{enumerate}
	\item a single coalescent event may occur at a given time, with no other event taking place simultaneously, and where blocks of different types may participate in the coalescence; or
	\item blocks may undergo simultaneous type switching without coalescing.
\end{enumerate}
In particular, using Proposition \ref{th:cerwCharact}, we give an explicit integral representation for the rates governing the latter kind of event.
This complements the results of \cite{flammGenealogyMultitypeCannings2025} and \cite{daipraMultitypeXcoalescentsStructured2025}, and contribute to the broader effort of studying multitype coalescent processes in the direction of integral representations of the rates governing their dynamics.

\section{Examples of dice processes}
Theorem \ref{th:cerwCharact} gives us a thorough description on the rates that govern a family of consistent and exchangeable Markov chains.
In particular, expression \eqref{eq:transitionCharact} tells us that the rates are composed of an independent part for each particle, given by the array $A$, and a coordinated part driven by $\nu$.

In this section we will give various examples of different dice processes in order to show that it is a quite flexible model, and not restricted to the coalescent setup we considered in the previous section.
The first two examples are the extreme cases of total independence and total dependence between the Markov chains.
Afterwards we will construct a class of dice processes such that the de Finetti measure of one of these dice processes is a \emph{stochastic exchange model}, recently introduced in \cite{kimSpectralGapKMP2025} and \cite{caputoUniversalCutoffPhenomenon2025}.
Then we will provide an example that extends the \emph{binomial splitting process} studied in \cite{quattropaniMixingAveragingProcess2023}, and whose de Finetti measure extends the \emph{block averaging process} studied in \cite{caputoRepeatedBlockAverages2024}.
Finally, we present the Dirichlet splitting process, the harmonic multinomial splitting process, and the discrete multiagent instant exchange model.
The reason to introduce these examples of dice processes is beacuse their de Finetti measures provide natural extensions of the Kipnis--Marchioro--Presutti model, the harmonic process and immediate exchange model studied in \cite{kimSpectralGapKMP2025}.
In order to present the examples, we will recall that for each $i, j \in [d]$, $E_{ij}$ is the elementary matrix with a $1$ at position $(i, j) \in [d]^2$ and zero everywhere else.

\paragraph{Independent Markov chains}
As the first example we consider the case of a countable collection of independent Markov chains with a common collection of rates $A = (a_{ij})_{i \in [d], j \in [d] \setminus \{i\}}$, where $a_{ij}$ is the rate of any Markov chain goes from state $i$ to state $j$.
By the assumed independence, this collection is obviously consistent.
Moreover, because of the assumption on the rates, it is clear that the collection is partially exchangable.
In this case the dice process has $\nu \equiv 0$, as there is no coordination.

\paragraph{Totally dependent Markov Chains}
In constrast to the first example, in this second example we consider that all of the Markov chains always jump at the same time, and at a jump time all of the Markov chains that are at the same state jump to the same place.
Thus, we first note that there is no independent component, meaning that $a_{ij} = 0$ for all $i \in [d]$ and $j \in [d] \setminus \{i\}$.
Secondly, that all of the Markov chains at a state jump to the same place means that $\nu$ has support on the set of matrices
\[
	\Bigl\{ \sum_{i = 1}^d E_{i f(i)} : f \text{ is a function from } [d] \text{ to } [d] \Bigr\} \,.
\]
Namely, if for $f : [d] \to [d]$ we define the matrix
\[
	U^{(f)} := \sum_{i = 1}^d E_{i f(i)}\,,
\]
then
\[
	\nu(dU) = \sum_{f : [d] \to [d]} c_f \delta_{U^{(f)}}(dU) \,,
\]
where $(c_f)_{f : [d] \to [d]}$ is a collection of nonnegative constants.
Intuitively, at rate $c_f$, all of the Markov chains at state $i$ move to state $f(i)$ for each $i \in [d]$.
In this case, the rate at which $X_1$ moves from state $i$ to state $j$ is
\[
	\int_{\Delta_{d-1}^d} u_{ij} \nu(dU) = \sum_{f : [d] \to [d]} c_f 1_{\{f(i) = j\}} < \infty \,.
\]

Having taken care of the most extreme, in terms of independence and dependence, of the examples, we will consider that $a_{ij} = 0$ for all $i \in [d]$ and all $j \in [d] \setminus \{i\}$ for the rest of the section.
This means that in the examples coordination is the force that drives the processes.
In addition we will consider a given fixed vector $\eta \in (0, \infty)^d$.

\paragraph{Discrete exchange models}
The stochastic exchange model introduced in \cite{kimSpectralGapKMP2025} is a Mar\-kov process $R = (R(t))_{t \geq 0}$ with values in $\Delta_{d-1}$ with infinitesimal generator defined by its action on any $f \in C(\Delta_{d-1})$ by
\[
	\mathcal{A}_{\mathrm{sem}} f(r)
	= \sum_{i = 1}^d \sum_{j \in [d] \setminus \{i\}} \int_{[0, 1]^2} \bigl( f\bigl(r + \bigl( (1 - s) r_i - (1 - v) r_j \bigr) (e_j - e_i) \bigr) - f(r) \bigr) \beta_{ij}(ds, dv) \,,
\]
where for each $i, j \in [d]$ with $i \neq j$, $\beta_{ij}$ is a measure on $[0, 1]^2$ that satisfies the integrability condition $\int_{[0, 1]^2} ((1-s) + (1-v)) \beta_{ij}(ds, dv) < \infty$.
To define the discrete exchange model consider the collection of measures $(\beta_{ij})_{i \in [d], j \setminus \{i\}}$ on $[0, 1]^2$ and define
\[
	\nu_{\mathrm{sem}}(dU) = \sum_{i = 1}^d \sum_{j \in [d] \setminus \{i\}} \beta_{ij}(ds, dv) \delta_{U^{(i,j,s,v)}}(dU) \,,
\]
where
\[
	U^{(i,j,s,v)} = \sum_{k \in [d] \setminus \{i, j\}} E_{kk} + (s E_{ii} + (1-v) E_{ji}) + ((1 - s) E_{ij} + v E_{jj}) \,.
\]
In this case, any one of the Markov chains goes from state $i$ to state $j$, with $i \neq j$, at rate
\[
	\int_{\Delta_{d-1}^d} u_{ij} \nu_{\mathrm{sem}}(dU)
	= \int_{[0, 1]^2} (1 - s) \beta_{ij}(ds, dv) + \int_{[0, 1]^2} (1 - v) \beta_{ji}(ds, dv) < \infty \,.
\]
By the latter inequality it transpires that we can define a dice process with parameter $\nu_{\mathrm{sem}}$.

In particular, what is interesting about this process is that its de Finetti measure is the stochastic exchange model.
To see this, note that
\[
	U^{(i,j,s,v)} = \sum_{k = 1}^d E_{kk} + [ (s - 1) E_{ii} + (1 - v) E_{ji} ] + [(1 - s) E_{ij} + (v-1) E_{jj}] \,,
\]
so
\[
	(U^{(i,j,s,v)})^T = \sum_{k = 1}^d E_{kk} + [ (s - 1) E_{ii} + (1 - v) E_{ij} ] + [(1 - s) E_{ji} + (v-1) E_{jj}] \,.
\]
Therefore, for any $r \in \Delta_{d-1}$,
\begin{align*}
	(U^{(i,j,s,v)})^T r
	 & = r + (-(1 -s) r_i + (1 -v) r_j) e_i + ((1 - s) r_i - (1 - v) r_j) e_j \\
	 & = r + ((1 - s) r_i - (1 - v) r_j) (e_j - e_i)\,.
\end{align*}
Thus, by considering \eqref{eq:generatorFinetti} with $\nu_{\mathrm{sem}}$ we note that
\begin{align*}
	\mathcal{A} f(r)
	 & = \int_{\Delta_{d-1}^d} (f(U^T r) - f(r)) \nu_{\mathrm{sem}}(dU)                                                          \\
	 & = \sum_{i = 1}^d \sum_{j \in [d] \setminus \{i\}} \int_{[0, 1]^2} ( f((U^{(i, j, s, v)})^T r) - f(r) ) \beta_{ij}(ds, dv) \\
	 & = \mathcal{A}_{\mathrm{sem}} f(r)\,,
\end{align*}
proving that the de Finetti measure of our dice process is the stochastic exchange model.

\paragraph{Multinomial splitting process}
For this example consider a collection of nonnegative constants $(c_J)_{J \subset [d], \# J \geq 2}$.
At rate $c_J$, the set $J$ is chosen.
Given the set $J$, for each $j \in J$, all of the Markov chains that are at any state $i \in J$ jump to state $j \in J$ with probability $\eta_j / \eta(J)$, where $\eta(J) = \sum_{j \in J} \eta_j$.
Using this description, the measure $\nu$ is seen to be
\[
	\nu_{\mathrm{msp}}(dU) = \sum_{J \subset [d] : \# J \geq 2} c_J \delta_{U^{(J)}}(dU) \,,
\]
where
\[
	U^{(J)} = \sum_{k \in [d] \setminus J} E_{kk} + \sum_{i \in J} \sum_{j \in J} \frac{\eta_j}{\eta(J)} E_{ij} \,.
\]
Under this measure,
\[
	\int_{\Delta_{d-1}^d} u_{ij} \nu(dU)
	= \sum_{J \subset [d] : J \supset \{i, j\}} c_J \frac{\eta_j}{\eta(J)} < \infty
\]
for each $i, j \in [d]$ such that $i \neq j$.
Thus it is a valid measure for a dice process which we call multinomial splitting process.
By noting that for each set of indices $J$,
\[
	(U^{(J)})^T = \sum_{k \in [d] \setminus J} E_{kk} + \sum_{j \in J} \frac{\eta_j}{\eta(J)} \sum_{i \in J} E_{ji} \,,
\]
we see that the de Finetti measure of the multinomial splitting process $R$ is such that if at a jump $\tau$, the set $J \subset [d]$ with $\# J \geq 2$ is chosen, then
for each $i \in J$, $R_{i}(\tau-)$ is replaced by $(\eta_{i} / \eta(J)) \sum_{j \in J} R_{j}(\tau-)$.
In particular, if $\eta$ is a multiple of the vector with all entries equal to $1$, the de Finetti process provides a continuous version of the block average process studied in \cite{caputoRepeatedBlockAverages2024}.
In a similar fashion, if $c_J \textcolor{blue}{=} 0$ for all $J \subset [d]$ such that $\# J \geq 3$, we obtain the binomial splitting process and averaging process, as the dice process and de Finetti measure respectively, studied in \cite{quattropaniMixingAveragingProcess2023}.

\paragraph{Dirichlet splitting process}
This example could be considered a randomized version of the multinomial splitting process.
Take the collection of nonnegative constants $(c_J)_{J \subset [d], \# J \geq 2}$.
Now at rate $c_J$, the set of indices $J$ is chosen, and a collection of independent random variables $(W_j)_{j \in J}$ is sampled such that $W_j$ has gamma distribution with form parameter $\eta_j$ and rate parameter $1$.
Then, each Markov chain that is at any state $i \in J$ moves to state $j \in J$ with probability $W_j / W(J)$, where $W(J) := \sum_{j \in J} W_j$.
By the relation between the gamma distribution and the Dirichlet distribution, in this case
\[
	\nu_{\mathrm{dsp}}(dU) = \sum_{J \subset [d] : \# J \geq 2} c_J \frac{\Gamma(\eta(J))}{\prod_{j \in J} \Gamma(\eta_j)} \prod_{j \in J} s_j^{\eta_j - 1} \lambda_{\Delta_J}(ds) \delta_{U^{(J, s)}}(dU) \,,
\]
where $\lambda_{\Delta_J}$ is the Lebesgue measure on the set $\Delta_J := \{s \in \Delta_{d - 1} : s_k = 0\ \forall k \in [d] \setminus J\}$, and
\[
	U^{(J, s)} = \sum_{k \in [d] \setminus J} E_{kk} + \sum_{i \in J} \sum_{j \in J} s_{j} E_{ij} \,.
\]
We now have that for any $i, j \in [d]$ with $i \neq j$,
\[
	\int_{\Delta_{d-1}} u_{ij} \nu_{\mathrm{dsp}}(dU)
	= \sum_{J \subset [d] : J \supset \{i, j\}} c_J \frac{\eta_j}{\eta(J)} < \infty\,.
\]
Thus we can define a dice process which we call the Dirichlet splitting process.
For this particular model, the de Finetti measure corresponds to a natural extension of the Kipnis--Marchioro--Presutti model studied by \cite{kimSpectralGapKMP2025}, which is a closed system version of the model studied in \cite{masiHiddenTemperatureKMP2024},  in the sense that we can consider averaging of more than two sites at the same time.

\paragraph{Harmonic splitting process}
For this model we will consider a collection of nonnegative numbers $(c_{iJ})_{i \in [d], J \subset [d] \setminus \{i\}, \# J \geq 1}$, and let us put
\[
	\nu_{\mathrm{hsp}}(dU)
	= \sum_{i = 1}^d \sum_{J \subset [d] \setminus \{i\} : \# J \geq 1} c_{iJ} \frac{s^{\eta_i - 1}}{1 - s} ds \delta_{U^{(i, J, s)}}(dU) \,,
\]
where
\[
	U^{(i, J, s)} = \sum_{k \in [d] \setminus (J \cup \{i\})} E_{kk} + s E_{ii} + \frac{1 - s}{\# J} \sum_{j \in J} E_{ij} \,.
\]
Informally, at rate $c_{iJ} (1 - s)^{-1} s^{\eta_i - 1}$ we choose a state $i \in [d]$, set $J \subset [d]$ and probability $s \in (0, 1)$.
Having chosen these, all of the Markov chains in state $i$ stay at the state with probability $s$ and move with probability $1 - s$.
If they move, then the probability that it moves to any state indexed by $J$ is equal.
Now we have that
\[
	\int_{\Delta_{d-1}^d} u_{ij} \nu_{\mathrm{hsp}}(dU)
	= \sum_{J \subset [d] \setminus \{i\} : \# J \geq 1} \frac{c_{iJ}}{\eta_i \# J} < \infty
\]
for any $i, j \in [d]$ such that $i \neq j$.
Thus we can define a dice process whose de Finetti measure is an extension of the harmonic process presented in \cite{kimSpectralGapKMP2025}.

\paragraph{Discrete multiagent instant exchange model}
Let us present a final model.
For this one we will consider another parameter $\kappa \in (0, \min_{i \in [d]} \eta_i)$.
Take $(c_J)_{J \subset [d], \# J \geq 2}$ as in the Multinomial splitting process, and set
\[
	\nu_{\mathrm{IEM}}(dU)
	= \sum_{J \subset [d] : \# J \geq 2} c_J \prod_{i \in J} \frac{\Gamma(\eta_i)}{\Gamma(\eta_i - \kappa) \Gamma(\kappa / \#J)^{(\# J) - 1}} s_{ii}^{\eta_i - 1} \prod_{j \in J \setminus \{i\}} s_{ij}^{(\kappa / \#J) - 1}  \lambda_{\Delta_J}(ds_i) \delta_{U^{(J, S)}}(dU) \,,
\]
where $S = (s_{ij})_{i, j \in J}$,
\[
	U^{(J, S)} = \sum_{k \in [d] \setminus \{J\}} E_{kk} + \sum_{i \in J} \sum_{j \in J} s_{ij} E_{ij} \,.
\]
This represents that at rate $c_J$ the set of indices $J$ is chosen, and when it is chosen, the probability that a Markov chain moves from state $i$ to $j$ is $s_{ij}$, where $s_i := (s_{ij})_{j \in J}$ is sampled according to a Dirichlet distribution.
It is readily seen that
\[
	\int_{\Delta_{d-1}^d} u_{ij} \nu(dU) = \sum_{J \subset [d] : J \supset \{i, j\}} c_J \frac{\kappa}{\eta_i \#J} < \infty \,,
\]
so a dice process with directing measure $\nu_{\mathrm{IEM}}(dU)$ exists.
We call it \emph{discrete multiagent instant exchange model} because its de Finetti measure is a version of the instant exchange model where more than two agents can be involved in the exchange, see section 2 in \cite{kimSpectralGapKMP2025} and references therein.

\section{Proofs}
\label{sec:diceProcess}

In this section we provide the remaining proofs.
We start with the proofs of Lemmas \ref{lem:consistencyArray} and \ref{lem:bkRatesExplicit} dealing with consistency of the rates of dice processes.
In Section \ref{subsec:proofCharacterizationDiceProcesses} we will prove Theorem \ref{thm:pExDP} characterising the partially exchangeable consistently Markovian processes as dices processes.
Finally in Section \ref{subsec:proofCoalescentCharacterization} we prove the characterisation of multitype-$\Lambda$-coalescents with multiple switching, Theorem \ref{thm:pmLambda}.
We defer the proofs of the de Finetti process and the moment duality result (which are somewhat technical, but following fairly standard methods) to Appendix \ref{Prop_2.2}.
\subsection{Proof: Characterisation of dice processes}

\paragraph{Proof of Lemma  \ref{lem:consistencyArray}}
Note that is suffices to establish that $R_n X^{(n + 1)}$ has the same law as $X^{(n)}$ if and only if \eqref{eq:consistencyEq} holds because  of the equality $R_m X^{(n)} = R_m R_{m+1} \cdots R_{n-1} X^{(n)}$ and an induction argument.
Equation \eqref{eq:consistencyEq} appears as a condition, in terms of transition rates, for a function---the restriction $R_n$--- of a finite state Markov Chain---$X^{(n + 1)}$---to be Markovian with the specified rates.
Let us carry out the argument explicitly.

Suppose that $R_n X^{(n + 1)}$ undergoes a $(b, K)$-change at time $t \in [0, \infty)$.
Then, assuming that $X_{n + 1}^{(n + 1)}(t-) = j$
, $X^{(n + 1)}$ can undergo any of the $(b + e_j, K + E_{jl})$-changes indexed by $l \in [d]$.
Thus the rate at which we observe a $(b, K)$-change in $R_n X^{(n + 1)}$ is given by the sum of the $d$ possible $(b + e_j, K + E_{jl})$-changes that $X^{(n + 1)}$ might observe, which is the right hand side of \eqref{eq:consistencyEq}.
Consequently we deduce that $X^{(n)}$ and $R_n X^{(n + 1)}$ have the same law if and only if \eqref{eq:consistencyEq} holds. \hfill $\Box$

\paragraph{Proof of Lemma \ref{lem:bkRatesExplicit}}
Assume first that \eqref{eq:transitionCharact} holds.
Then, by Lemma \ref{lem:consistencyArray} it suffices to prove that $\gamma_{b, K}$ defined by \eqref{eq:transitionCharact} satisfies \eqref{eq:consistencyEq}.
It is sufficient to consider the following two cases:
\begin{enumerate}
	\item $a_{i\ell} = 0$ for all $i, \ell$,
	\item $\nu \equiv 0$,
\end{enumerate}
as the general case follows by considering the sum.
In the first case, the right hand side of \eqref{eq:consistencyEq} becomes
\begin{align*}
	\sum_{l = 1}^d \int_{\Delta_{d-1}^d} \prod_{i = 1}^d \prod_{\ell = 1}^d u_{i \ell}^{k_{i\ell} + \delta_{ij} \delta_{l \ell}} \nu(dU)
	 & = \int_{\Delta_{d-1}^d} \prod_{i = 1}^d \prod_{\ell = 1}^d u_{i \ell}^{k_{i \ell}} \sum_{l = 1}^d u_{jl} \nu(dU) \\
	 & = \int_{\Delta_{d-1}^d} \prod_{i = 1}^d \prod_{\ell = 1}^d u_{i \ell}^{k_{i \ell}} \nu(dU) = \gamma_{b, K} \,.
\end{align*}
For the second case consider $b \in \mathcal{S}_{d-1}(n)$ and $K \in \mathcal{S}_{d-1,d}(b) \setminus \{\diag(b)\}$, and observe that $b + e_{j} \in \mathcal{S}_{d-1}(n+1)$ and $K + E_{jl} \in \mathcal{S}_{d-1,d}(b + e_j) \setminus \{ \diag(b + e_j) \}$ for each $j, l \in [d]$.
In this case, the right hand side of \eqref{eq:consistencyEq} is equal to
\begin{align*}
	\sum_{l = 1}^d \gamma_{b + e_j, K + E_{jl}}
	 & = \sum_{i = 1}^d \sum_{\ell \in [d] \setminus \{i\}} a_{i \ell} \Bigl( \sum_{l = 1}^d 1_{\{K = \diag(b) + E_{jj} - E_{jl} - E_{ii} + E_{i\ell}\}} \Bigr) \,.
\end{align*}
Note that when $j \neq i$, $\diag(b) + E_{jj} - E_{jl} - E_{ii} + E_{i\ell}$ has a negative $(j,l)$ entry if and only if $l \neq j$.
Similarly, when $j = i$, $\diag(b) + E_{jj} - E_{jl} - E_{ii} + E_{i \ell}$ has a negative entry $(i, l)$ entry whenever $l \neq i, \ell$.
Thus, we deduce the first equality in
\begin{align*}
	\sum_{l = 1}^d \gamma_{b + e_j, K + E_{jl}}
	 & = \sum_{i \in [d] \setminus \{j\}} \sum_{\ell \in [d] \setminus \{i\}} a_{i \ell} 1_{\{K = \diag(b) - E_{ii} + E_{i \ell}\}}           \\
	 & \quad + \sum_{\ell \in [d] \setminus \{j\}} a_{j \ell} \bigl( 1_{\{K = \diag(b)\}} + 1_{\{K = \diag(b) - E_{jj} + E_{j \ell}\}} \bigr) \\
	 & = \sum_{i = 1}^d \sum_{\ell \in [d] \setminus \{j\}} a_{i \ell} 1_{\{K = \diag(b) - E_{ii} + E_{i \ell}\}}
	= \gamma_{b, K} \,,
\end{align*}
whereas the second equality is due to the fact that $1_{\{K = \diag(b)\}} = 0$.
This proves that \eqref{eq:consistencyEq} holds, meaning that $\gamma$ is consistent.

In the other direction assume that $\gamma$ is consistent.
By Lemma \ref{lem:consistencyArray} we know that $\gamma$ must satisfy \eqref{eq:consistencyEq}.
Given $i\in [d]$ and $j \in [d] \setminus \{i\}$, for $b \in \mathbb{N}_0^d$ and $K \in \mathcal{S}_{d - 1}(b) \setminus \{\diag(b)\}$, we define $\tilde{\gamma}_{b, K}^{(ij)} := \gamma_{b + e_i, K + E_{ij}} / \gamma_{e_i, E_{ij}}$, which is a number within $[0, 1]$ in view of \eqref{eq:consistencyEq}.
Then $\tilde{\gamma}^{(ij)} := (\tilde{\gamma}_{b, K}^{(ij)} : b \in \mathbb{N}_0^d, K \in \mathcal{S}_{d-1,d}(b))$ satisfies \eqref{eq:consistencyEq} with $\tilde{\gamma}_{0, 0}^{(ij)} = 1$.
Let us construct an array of random variables $Z^{(ij)} = (Z_{l, \ell}^{(ij)})_{l \in [d], \ell \in \mathbb{N}}$ with values in $[d]$ on an auxiliary probability space $(\Omega', \mathbb{P}')$ using $\tilde{\gamma}^{(ij)}$.
Consider $b \in \mathbb{N}_0^d$ and $K \in \mathcal{S}_{d-1,d}(b)$ given.
For each $l \in [d]$ with $b_l > 0$ let $x^{(l)}$ be any element of $[d]^{b_l}$ such that $k_{l q} = \#\{ m \in [b_l] : x_m^{(l)} = q \}$ for every $q \in [d]$.
If $b_l = 0$ set $x^{(l)}$ as $\partial$, any arbitrary point as $[d]^{0} = \varnothing$.
Now define
\begin{equation} \label{eq:exchArray1}
	\mathbb{P}'\Bigl( \bigcap_{l \in [d] : b_l > 0} \bigcap_{\ell \in [b_l]} \{ Z_{l \ell}^{(ij)} = x_{\ell}^{(l)} \} \Bigr) := \tilde{\gamma}_{b, K}^{(ij)} \,,
\end{equation}
where we understand $\bigcap_{\varnothing} \cdot$ as $\Omega'$.
The explicit construction is given in Appendix \ref{app:PExArray}.
For any choice of finite permutations $\sigma_1, \ldots, \sigma_d : \mathbb{N} \to \mathbb{N}$ put $Z_{\bm{\sigma}}^{(ij)} := (Z_{l \sigma_l(\ell)}^{(ij)})_{l \in [d], \ell \in \mathbb{N}}$.
It is then not difficult to see that for any collection $\sigma_1, \ldots, \sigma_d$ of finite permutations  of $\mathbb{N}$, $Z^{(ij)}$ and $Z_{\bm{\sigma}}^{(ij)}$ have the same distribution.

Before proceeding, let us note that for each $l \in [d]$, $(Z_{l\ell}^{(ij)})_{\ell \in \mathbb{N}}$ is an exchangable sequence on its own right.
Thus, there exists a set of probability zero
outside of which
\[
	\xi_{\ell}^{(ijl)} = \lim_{n \to \infty} \frac{ \#\{m \in [n] : Z_{lm}^{(ij)} = \ell\} }{n}
\]
is well defined for every $l, \ell \in [d]$.
For each $l \in [d]$, the random vector $\xi^{(ijl)} = (\xi_1^{(ijl)}, \ldots, \xi_d^{(ijl)})$ obtained through this procedure can be seen as the directing measure of the sequence $(Z_{l\ell}^{(ij)})_{\ell \in \mathbb{N}}$.
De Finetti's theorem for partially exchangeable arrays, see for instance Corollary 1.7 in \cite{kallenbergProbabilisticSymmetriesInvariance2005} or Corollary 3.9 in \cite{aldousExchangeabilityRelatedTopics1985},
allows us then to deduce that
\[
	\mathbb{P}'\Bigl( \bigcap_{l \in [d]} \bigcap_{\ell \in [b_l]} \{ Z_{l \ell}^{(ij)} = x_{\ell}^{(l)} \} \Bigr)
	= \mathbb{E}'\Bigl[ \prod_{l \in [d]} \prod_{\ell \in [b_l]} \xi_{x_{\ell}^{(l)}}^{(ijl)} \Bigr] \,.
\]
Hence, there exists a probability measure $\tilde{\nu}_{ij}$ on $\Delta_{d-1}^d$, which is undestood as the joint distribution of the directing measures, such that
\begin{equation} \label{eq:exchArray2}
	\mathbb{P}'\Bigl( \bigcap_{l \in [d]} \bigcap_{\ell \in [b_l]} \{ Z_{l \ell}^{(ij)} = x_{\ell}^{(l)} \} \Bigr)
	= \int_{\Delta_{d-1}^d} \prod_{l = 1}^d \prod_{\ell \in [b_l]} u_{l x_{\ell}^{(l)}} \tilde{\nu}_{ij}(dU) \,,
\end{equation}
for any $b \in \mathbb{N}_0^d$ and any $x^{(l)} \in [d]^{b_l}$ with $l \in [d]$ (putting $x^{(l)} = \partial$ if $b_l = 0$).
In the previous expression we understand $\prod_{\varnothing} \cdot = 1$.

By \eqref{eq:exchArray1} and \eqref{eq:exchArray2} it follows that
\[
	\tilde{\gamma}_{b, K}^{(ij)}
	= \int_{\Delta_{d-1}^d} \prod_{l = 1}^d \prod_{\ell = 1}^d u_{l \ell}^{k_{l \ell}} \tilde{\nu}_{ij}(dU) \,.
\]
Before proceeding let us note that if $U = \Id$, the identity matrix in $\Delta_{d-1}^d$, then $\prod_{l = 1}^d \prod_{\ell = 1}^d u_{l \ell}^{k_{l \ell}} > 0$ if and only if $K = \diag(b)$, allowing us to further deduce that
\[
	\tilde{\gamma}_{b, K}^{(ij)}
	= \tilde{\nu}_{ij}(\{\Id\}) 1_{\{K = \diag(b)\}}
	+ \int_{\Delta_{d-1}^d} \prod_{l = 1}^d \prod_{\ell = 1}^d u_{l\ell}^{k_{l\ell}} \tilde{\nu}_{ij}^0(dU) \,,
\]
where $\tilde{\nu}_{ij}^0(A) := \tilde{\nu}_{ij}(A \setminus \{\Id\})$.

Thus, defining $\nu_{ij} = \gamma_{e_i, E_{ij}} \tilde{\nu}_{ij}$ and $\nu_{ij}^0 = \gamma_{e_i, E_{ij}} \tilde{\nu}_{ij}^0$, we deduce that for $b \in \mathbb{N}_0^d$ and $K \in \mathcal{S}_{d-1,d}(b)$,
\begin{equation} \label{eq:gammaRep1}
	\gamma_{b + e_i, K + E_{ij}}
	= \nu_{ij}(\{\Id\}) 1_{\{K = \diag(b)\}}
	+ \int_{\Delta_{d-1}^d} \prod_{l = 1}^d \prod_{\ell = 1}^d u_{l\ell}^{k_{l\ell}} \nu_{ij}^0(dU) \,.
\end{equation}
Consider now $i_1, i_2 \in [d]$, $j_1 \in [d] \setminus \{i_1\}$ and $j_2 \in [d] \setminus \{i_2\}$ such that $i_1 \neq i_2$ or $j_1 \neq j_2$.
For $b \in \mathbb{N}_0^d$ and
$K \in \mathcal{S}_{n-1}$, using \eqref{eq:gammaRep1} we deduce
\begin{align*} \MoveEqLeft
	\nu_{i_1j_1}(\{\Id\}) 1_{\{K = \diag(b)+E_{i_2i_2}-E_{i_2j_2}\}}
	+ \int_{\Delta_{d-1}^d} \prod_{l = 1}^d \prod_{\ell = 1}^d u_{l \ell}^{k_{l \ell}} u_{i_2 j_2} \nu_{i_1j_1}^0(dU) \\
	 & =\gamma_{b + e_{i_1} + e_{i_2}, K + E_{i_1j_1} + E_{i_2j_2}}                                                   \\
	 & =\nu_{i_2j_2}(\{\Id\}) 1_{\{K = \diag(b)+E_{i_1i_1}-E_{i_1j_1}\}}
	+ \int_{\Delta_{d-1}^d} \prod_{l = 1}^d \prod_{\ell = 1}^d u_{l \ell}^{k_{l \ell}} u_{i_1 j_1} \nu_{i_2 j_2}^0(dU) \,.
\end{align*}
Note that because $i_1 \neq j_1$, $1_{\{K = \diag(b) + E_{i_1 i_1} - E_{i_1 j_1}\}} = 1_{\{K + E_{i_1 j_1} = \diag(b) + E_{i_1 i_1}\}} = 0$.
Indeed, the latter equality is due to the fact that $\diag(b) + E_{i_1 i_1}$ is a diagonal matrix while $K + E_{i_1 j_1}$ is not.
Similarly we see that $1_{\{K = \diag(b) + E_{i_2 i_2} - E_{i_2 j_2}\}} = 0$.
Hence, we get
\begin{align*}
	\int_{\Delta_{d-1}^d} \prod_{l = 1}^d \prod_{\ell = 1}^d u_{l \ell}^{k_{l \ell}} u_{i_2 j_2} \nu_{i_1j_1}^0(dU)
	 & = \gamma_{b + e_{i_1} + e_{i_2}, K + E_{i_1j_1} + E_{i_2j_2}}
	= \int_{\Delta_{d-1}^d} \prod_{l = 1}^d \prod_{\ell = 1}^d u_{l \ell}^{k_{l \ell}} u_{i_1 j_1} \nu_{i_2 j_2}^0(dU) \,,
\end{align*}
allowing us to deduce $u_{i_2 j_2} \nu_{i_1 j_1}^0(dU) = u_{i_1 j_1} \nu_{i_2 j_2}^0(dU)$.
This has two consequences:
\begin{itemize}
	\item First, $\nu_{ij}^0$ does not charge the set $\{u_{ij} = 0\}$.
	      Indeed, by an easy application of the previous result we obtain the third equality in
	      \begin{align*}
		      \nu_{ij}^{0}(\{u_{ij} = 0\})
		       & = \int_{\{u_{ij} = 0\}} \sum_{l = 1}^d u_{1l} \nu_{ij}^0(dU)                                               \\
		       & = \int_{\{u_{ij} = 0\}} u_{11} \nu_{ij}^0(dU) + \sum_{l = 2}^d \int_{\{u_{ij} = 0\}} u_{ij} \nu_{1l}^0(dU) \\
		       & = \int_{\{u_{ij} = 0\}} u_{11} \nu_{ij}^0(dU) \,.
	      \end{align*}
	      Proceeding similarly, we can deduce that for any $\beta \in \mathbb{N}$,
	      \[
		      \nu_{ij}^0(\{u_{ij} = 0\}) = \int_{\{u_{ij} = 0\}} \bigl( \prod_{l = 1}^d u_{ll} \bigr)^\beta \nu_{ij}^0(dU) \,.
	      \]
	      By letting $\beta \to \infty$, using the fact that $\nu_{ij}^0(\{\Id\}) = 0$, the Dominated Convergence Theorem implies that
	      \[
		      \nu_{ij}^0(\{u_{ij} = 0\}) = \lim_{\beta \to \infty} \int_{\{u_{ij} = 0\}} \bigl( \prod_{l = 1}^d u_{ll} \bigr)^\beta \nu_{ij}^0(dU) = 0 \,.
	      \]
	\item Second, $u_{i_1 j_1}^{-1} \nu_{i_1 j_1}^0(dU) = u_{i_2 j_2}^{-1} \nu_{i_2 j_2}^0(dU)$ on $\{u_{i_1 j_1} > 0\} \cap \{u_{i_2 j_2} > 0\}$ for $i_1, i_2 \in [d]$, $j_1 \in [d] \setminus \{i_1\}$ and $j_2 \in [d] \setminus \{i_2\}$.
	      Surely, if $A \subset \{u_{i_1 j_1} > 0\} \cap \{u_{i_2 j_2} > 0\}$ is a measurable set, then
	      \[
		      \int_A u_{i_1 j_1}^{-1} \nu_{i_1 j_1}^0(dU)
		      = \int_A u_{i_1 j_1}^{-1} u_{i_2 j_2}^{-1} u_{i_2 j_2} \nu_{i_1 j_1}^0(dU)
		      = \int_A u_{i_1 j_1}^{-1} u_{i_2 j_2}^{-1} u_{i_1 j_1} \nu_{i_2 j_2}^0(dU)
		      = \int_A u_{i_2 j_2}^{-1} \nu_{i_2 j_2}^0(dU) \,.
	      \]
\end{itemize}
By the latter point and the fact that $\bigcup_{i = 1}^d \bigcup_{j = 1}^d \{u_{ij} > 0\} = \Delta_{d-1}^d$, we may define a measure $\nu$ on $\Delta_{d-1}^d$ such that $\nu(dU) = u_{ij}^{-1} \nu_{ij}^0(dU)$ on $\{u_{ij} > 0\}$.
Note then that if $b_i, k_{ij} > 0$ for some fixed $i \in [d]$ and $j \in [d] \setminus \{i\}$, an application of \eqref{eq:gammaRep1} gives us the second equality in
\begin{equation} \label{eq:transitionCharactAux}
	\begin{split}
		\gamma_{b, K}
		 & = \gamma_{(b - e_i) + e_i, (K - E_{ij}) + E_{ij}}
		= \nu_{ij}(\{\Id\}) 1_{\{K = \diag(b) - E_{ii} + E_{ij}\}} + \int_{\Delta_{d-1}^d} \prod_{l = 1}^d \prod_{\ell = 1}^d u_{l \ell}^{k_{l \ell}} u_{ij}^{-1} \nu_{ij}^{0}(dU)                                                \\
		 & = \sum_{l = 1}^d \sum_{\ell \in [d] \setminus \{l\}} \nu_{l \ell}(\{\Id\}) 1_{\{K = \diag(b) - E_{ll} + E_{l \ell}\}} + \int_{\Delta_{d - 1}^d} \prod_{l = 1}^d \prod_{\ell = 1}^d u_{l \ell}^{k_{l \ell}} \nu(dU) \,,
	\end{split}
\end{equation}
The last equality follows by definition of $\nu$, that $\nu_{ij}^0$ has no support on $\{u_{ij} = 0\}$, and the observation that the condition $k_{ij} > 0$ implies that if $l \neq i$ or $\ell \neq j$, then $(\diag(b) - E_{ll} + E_{l\ell})_{ij} = 0$, so $1_{\{K = \diag(b) - E_{ll} + E_{l \ell}\}} = 0$.
By defining $a_{ij} = \nu_{ij}(\{\Id\})$, expression \eqref{eq:transitionCharactAux} is precisely the same as \eqref{eq:transitionCharact}.
To finish the proof, note that for each $i \in [d]$ and $j \in [d] \setminus \{i\}$,
\[
	\int_{\Delta_{d-1}^d} u_{ij} \nu(dU)
	= \nu_{ij}^0(\Delta_{d-1}^d \setminus \{u_{ij} = 0\}) < \infty \,,
\]
so
\[
	\int_{\Delta_{d - 1}^d} \sum_{i = 1}^d (1 - u_{ii}) \nu(dU)
	= \sum_{i = 1}^d \sum_{j \in [d] \setminus \{i\}} \int_{\Delta_{d-1}^d} u_{ij} \nu(dU) < \infty \,,
\]
proving \eqref{eq:nuConstraint}. \hfill $\Box$

\subsection{Proof of Theorem \ref{thm:pExDP}}
\label{subsec:proofCharacterizationDiceProcesses}

Now that we have proven a characterization of dice processes in terms of processes that satisfy $(b, K)$ transitions of configurations, we are ready to prove Theorem \ref{thm:pExDP}, stating that the class of dice processes coincides with the class of partially exchangeable processes $Y^{(\infty)}$ that satisfy that $Y^{(n)} = R_n Y^{(\infty)}$ is a $[d]^n$-valued Markov chain for each $n \in \mathbb{N}$.
The fact that dice processes are, in general, only partially exchangeable is due to the initial condition, the value of the process at time zero, which is generally not exchangeable.
Therefore, it would seem that if we allow the initial condition to be exchangeable, then we may obtain full exchangeability of the processes.
This is what the second part of the statement says.

Let us begin the proof.
On one hand, Remark \ref{rmk:pExchOfDice1} states that a dice process is partially exchangeable.
On the other hand, suppose that $Y^{(\infty)}$ is partially exchangeable.
Let us first observe that $Y_m^{(\infty)} = \pi_m Y^{(\infty)}$ takes values in $D(\mathbb{R}_+, [d])$, the space of càdlàg functions with values in $[d]$.
Consider the initial configuration $\tilde{x} \in [d]^\infty$ given by $\tilde{x}_n = d$ if $n$ is divisible by $d$ and the remainder of $n$ when divided by $d$ otherwise; this is $\tilde{x} = (1, 2, \ldots, d, 1, 2, \ldots, d, \ldots)$.
Note that any other starting configuration $x \in [d]^\infty$ may be obtained by considering suitable projections of $\tilde{x}$.
Using said starting configuration, $\tilde{x}$, we will define an array $\tilde{Y} = (\tilde{Y}_{ij})_{i \in [d], j \in \mathbb{N}}$ by setting $\tilde{Y}_{ij} = Y_{(j-1) d + i}^{(\infty)}$.
It is then easy to see that if $\sigma_1, \ldots, \sigma_d : \mathbb{N} \to \mathbb{N}$ are finite permutations, then $(\tilde{Y}_{i \sigma_i(j)})_{i \in [d], j \in \mathbb{N}}$ has the same law as $\tilde{Y}$.
Indeed, for any $n \in \mathbb{N}$ let $i \in [d]$ and $j \in \mathbb{N}$ be the unique numbers such that $n = (j - 1) d + i$, and define $\sigma(n) := (\sigma_i(j) - 1) d + i$.
Then $\sigma$ is a finite permutation of $\mathbb{N}$ such that $\sigma \in \mathrm{Perm}_{\mathbf{A}^{\tilde{x}}}$, where $\mathbf{A}^{\tilde{x}}$ is the partition of $\mathbb{N}$ induced by $\tilde{x}$ (as in point (1.) of Remark \ref{rem:ex_dice_infinity}).
Moreover, the array obtained from $Y_{\sigma}^{(\infty)}$ is precisely $(\tilde{Y}_{i\sigma_i(j)})_{i \in [d], j \in \mathbb{N}}$, so the equality in law follows from the partial exchangeability of $Y^{(\infty)}$.
Therefore, we can apply Corollary 3.9 in \cite{aldousExchangeabilityRelatedTopics1985}  to deduce the existence of a probability measure $\alpha$ on $\mathcal{P}^d$, where $\mathcal{P}$ is the space of probability measures on $D(\mathbb{R}_+, [d])$, such that
\[
	\mathbb{P}_{\tilde{x}}\Bigl( \bigcap_{i = 1}^d \bigcap_{m_i = 1}^{n_i} \{\tilde{Y}_{i m_i} \in A_{i m_i}\} \Bigr) = \int_{\mathcal{P}^d} \prod_{i = 1}^d \prod_{m_i = 1}^{n_i} P_i(A_{i m_i}) \alpha(dP_1, \ldots, dP_d)
\]
is true for any $n \in \mathbb{N}_0^d$ and any measurable sets $A_{11}, \ldots, A_{1 n_1}, \ldots, A_{d1}, \ldots, A_{d n_d}$ of $D(\mathbb{R}_+, [d])$.
In terms of the sequence of processes $Y^{(\infty)}$ this is rewritten simply as
\[
	\mathbb{P}_{\tilde{x}}\Bigl( \bigcap_{m = 1}^n \{Y_m^{(\infty)} \in A_m\} \Bigr) = \int_{\mathcal{P}^d} \prod_{m = 1}^n P_{\tilde{x}_m}(A_m) \alpha(dP_1, \ldots, dP_d)
\]
for all $n \in \mathbb{N}$ and any measurable sets $A_1, \ldots, A_n$ of $D(\mathbb{R}_+, [d])$.
Moreover, by using suitable projections, it is readily seen that
\begin{equation} \label{eq:pExofY}
	\mathbb{P}_x\Bigl( \bigcap_{m = 1}^n \{ Y_m^{(\infty)} \in A_m \} \Bigr)
	= \int_{\mathcal{P}^d} \prod_{m = 1}^n P_{x_m}(A_m) \alpha(dP_1, \ldots, dP_d)
\end{equation}
holds for any $x \in [d]^{(\infty)}$, any $n \in \mathbb{N}$, and any measurable sets $A_1, \ldots, A_n$ of $D(\mathbb{R}_+, [d])$.
Taking any finite permutation $\sigma : \mathbb{N} \to \mathbb{N}$ and noting that
\[
	\int_{\mathcal{P}^d} \prod_{m = 1}^n P_{x_m}(A_m) \alpha(dP_1, \ldots, dP_d)
	= \int_{\mathcal{P}^d} \prod_{m = 1}^{\max\{\sigma(j) : j \in [n]\}} P_{x_{\sigma(m)}}(A_{\sigma(m)}) \alpha(dP_1, \ldots, dP_d)\,,
\]
where we put $A_j = D(\mathbb{R}_+, [d])$ for all $j > n$,
\eqref{eq:pExofY} entails the relation
\begin{equation} \label{eq:pExPermutationInitialValues}
	\mathbb{P}_x\Bigl( \bigcap_{m = 1}^n \{ Y_m^{(\infty)} \in A_m \} \Bigr)
	= \mathbb{P}_{x_{\sigma}}\Bigl( \bigcap_{m = 1}^n \{ Y_m^{(\infty)} \in A_{\sigma(m)} \} \Bigr) \,,
\end{equation}
for all $n \geq N$, where $N := \max\{j \in \mathbb{N} : \sigma(j) \neq j, \sigma(j + 1) = j + 1\}$.

Now consider $x, y \in [d]^\infty$, and define, for each $i \in [d]$, the set of indices $A_i^{x} := \{j \in \mathbb{N} : x_j = i\}$.
Let $\sigma : \mathbb{N} \to \mathbb{N}$ be any finite permutation such that $\sigma(A_i^{x}) = A_i^{x}$ for every $i\in[d]$, and consider $N$ as in the previous paragraph.
Consider the time of the first jump of $Y^{(N)} = R_N Y^{(\infty)}$ defined by $\tau_1(N) := \inf\{t \geq 0 : Y^{(N)}(t) \neq Y^{(N)}(0)\}$.
Note that this is the same as the first jump of $Y_\sigma^{(N)} = R_N Y_\sigma^{(\infty)}$.
By partial exchangeability we obtain, for any $y \in [d]^N$, the first equality in
\begin{equation} \label{eq:jumpPartial}
	\begin{split}
		\mathbb{P}\bigl( Y^{(N)}(\tau_1(N)) = R_N y, Y^{(N)}(0) = R_N x \bigr)
		 & = \mathbb{P}\bigl( Y_\sigma^{(N)}(\tau_1(N)) = R_N y, Y_\sigma^{(N)}(0) = R_N x \bigr)               \\
		 & = \mathbb{P}\bigl( Y^{(N)}(\tau_1(N)) = R_N y_{\sigma^{-1}}, Y^{(N)}(0) = R_N x_{\sigma^{-1}} \bigr) \\
		 & = \mathbb{P}\bigl( Y^{(N)}(\tau_1(N)) = R_N y_{\sigma^{-1}}, Y^{(N)}(0) = R_N x \bigr) \,,
	\end{split}
\end{equation}
whereas in the second equality $\sigma^{-1}$ is the inverse of $\sigma$, and the last equality follows because we assumed that $\sigma(A_i^{x}) = A_i^{x}$ for every $i \in [d]$.
By being $y$ arbitrary, from \eqref{eq:pExPermutationInitialValues} and \eqref{eq:jumpPartial}, it is clear that the rates of $Y^{(N)}$ satisfy \eqref{eq:pexRates}.
As $N$ was arbitrary, it follows that $Y^{(\infty)}$ must be a dice process, proving the first part of the statement.

Let us prove the last part, namely that exchangeability of a dice process $X^{(\infty)}$ is equivalent to the exchangeability of the initial value.
Clearly the exchangeability of $X^{(\infty)}$ implies the exchangeability of $X^{(\infty)}(0)$.
Let us assume then that $X^{(\infty)}(0)$ is exchangeable.
Consider $n, m \in \mathbb{N}$, $0 = t_0 < t_1 < \cdots < t_m$, and $y^{(0)}, y^{(1)}, \ldots, y^{(m)} \in [d]^n$ given.
Let $\mathbb{P}_{n, x}$ be the law of the $n$-dice process $X^{(n)} := R_n X^{(\infty)}$ started from $x$.
Then, by the Markov property of $X^{(n)}$,
\begin{equation} \label{eq:aux11}
	\mathbb{P}\Bigl( \bigcap_{j = 0}^m \{X^{(n)}(t_j) = y^{(j)}\} \Bigr)
	= \mathbb{P}(X^{(n)}(0) = y^{(0)}) \prod_{j = 1}^m \mathbb{P}_{n, y^{(j-1)}}(X^{(n)}(t_j - t_{j-1}) = y^{(j)}) \,.
\end{equation}
Now, for any permutation $\sigma : [n] \to [n]$, by Remark \ref{rmk:almostExchOfDice1} we have that
\begin{equation} \label{eq:aux22}
	\mathbb{P}_{n, y^{(j-1)}}(X^{(n)}(t_j - t_{j-1}) = y^{(j)})
	= \mathbb{P}_{n, y_{\sigma^{-1}}^{(j-1)}}(X^{(n)}(t_j - t_{j-1}) = y_{\sigma^{-1}}^{(j)})\,,
\end{equation}
while
\begin{equation} \label{eq:aux33}
	\mathbb{P}(X^{(n)}(0) = y^{(0)}) = \mathbb{P}(X^{(n)}(0) = y_{\sigma^{-1}})
\end{equation}
follows by the exchangeability of $X^{(\infty)}(0)$.
Therefore, from \eqref{eq:aux11}, \eqref{eq:aux22}, and \eqref{eq:aux33} we get the first equality in
\begin{align*}
	\mathbb{P}\Bigl( \bigcap_{j = 0}^m \{X^{(n)}(t_j) = y^{(j)}\} \Bigr)
	 & = \mathbb{P}\Bigl( \bigcap_{j = 0}^m \{X^{(n)}(t_j) = y_{\sigma^{-1}}^{(j)}\} \Bigr)
	= \mathbb{P}\Bigl( \bigcap_{j = 0}^m \{X_{\sigma}^{(n)}(t_j) = y^{(j)}\} \Bigr) \,,
\end{align*}
while the latter equality follows by noting that $x_\sigma = y$ if and only if $x = y_{\sigma^{-1}}$, for $x, y \in [d]^n$ and a permutation $\sigma : [n] \to [n]$.
Hence we deduce that $X^{(n)}$ and $X_\sigma^{(n)}$ have the same law for every permutation $\sigma : [n] \to [n]$.
Due to the fact that $n \in \mathbb{N}$ was chosen arbitrary and the consistency condition of $X^{(\infty)}$,
this implies that $X^{(\infty)}$ and $X_{\sigma}^{(\infty)}$ have the same law for any finite permutation $\sigma : \mathbb{N} \to \mathbb{N}$.
This finishes the proof.  \hfill $\Box$

\subsection{Proof of Theorem \ref{thm:pmLambda}}
\label{subsec:proofCoalescentCharacterization}

We now turn to the problem of characterizing the multitype coalescent in the Pitman sense.
Before Theorem \ref{thm:pmLambda} a class of M-$\Lambda$-MS-coalescents was constructed by superposing an independent M-$\Lambda$-coalescent and a dice process.
We now prove that such a decomposition always holds for any M-$\Lambda$-MS-coalescent.
By the discussion preceding Definition \ref{def:PMLcoal} it is clear that $\mathbf{\Pi}$ has independent coalescence and type transition mechanisms.

Let us now suppose that $\bm{\Pi}$ is a M-$\Lambda$-MS-coalescent, and let us consider any arbitrary, but fixed, starting condition $\bm{\pi} \in \mathbf{P}_d$ such that $\pi_i = \{i\}$ for every $i \in \mathbb{N}$, so we will work under the probability measure $\mathbb{P}^{\bm{\pi}}$.
We will write $\bm{\pi}(j)$ to stress that the first block of $\bm{\pi}$ is of type $j$ and $\bm{\pi}(jj)$ to stress that the first two blocks of $\bm{\pi}$ are of type $j$.
Recall that $X_n(t)$ is the type of the $n$-th block of $\bm{\Pi}$ at time $t \geq 0$, hence $X_n(0)=\mathfrak{t}_n$ in the notation of Section 2.5.
Define $T_1 := \inf\{t > 0 : X_1(t) \neq \mathfrak{t}_1 \}$ the first time that the type of block $1$ changes and $T_{\{1, 2\}} = \inf\{t > 0 : 2 \in \Pi_1(t)\}$ the first time that $1$ and $2$ are in the same block.
In what follows, we will use these random stopping times in a similar fashion that Schweinsberg uses $T_{\{1,2\}}$ in the proof of the ``only if'' part of Theorem 2 in \cite{schweinsbergCoalescentsSimultaneousMultiple2000}.
Exploiting exchangeability, we aim to show that the only possible transitions in the process are either of dice-type or of multitype-$\Lambda-$type (satisfying the respective consistency conditions), with no simultaneous occurrence of both.

Let us begin by considering only $T_1$.
Note that if $T_1 = \infty$ a.s. with respect to $\mathbb{P}^{\bm{\pi}(j)}$,
this would imply that the rates for a type transition or coalescence of blocks of type $j$ to other types are 0.  Thus, we may assume $T_1 < \infty$ a.s.
In this case define, for every $n$, $E_n$ as the event where the blocks containing $1, \ldots, n$ do not merge nor change type up to time $T_1-$; namely, that the first jump observed by the partition restricted to $[n]$ involves a type transition of the block containing $1$.
By the Markov consistency of $\mathbf{\Pi}$, $\mathbf{\Pi}\vert_{[n]}$ is a Markov chain with a finite state space and $E_n$ only depends on $\mathbf{\Pi}\vert_{[n]}$.
Thus, $\mathbb{P}^{\bm{\pi}(j)}(E_n) > 0$.
By the previous argument, we can define a random $d$-type partition $\mathbf{\Theta}^{(n)}$ of $[n]$ that has the same law as $\mathbf{\Pi}\vert_{[n]}(T_1)$ given $E_n$.
Moreover, the same consistency ensures that we can define a random $d$-type partition $\mathbf{\Theta}^{(\infty)}$ on $\mathbb{N}$ such that $\mathbf{\Theta}^{(\infty)}\vert_{[n]}$ has the same law as $\mathbf{\Theta}^{(n)}$.

Recall that the $n$-th block of a typed partition $\mathbf{\Gamma}$ is denoted by $\Gamma_n$, so the $n$-th block of $\mathbf{\Theta}^{(\infty)}$ is denoted by $\Theta^{(\infty)}_n$
In particular, the first block is denoted by $\Theta_1^{(\infty)}$ and it is such that $1 \in \Theta_1^{(\infty)}$.
For $k \neq j$ define $\mathsf{S}_k$ as the event
where the first block of $\mathbf{\Theta}^{(\infty)}$ is of type $k$.
In addition, define
\[
	\mathsf{C} = \bigcup_{n \geq 2} \{n \in \Theta_1^{(\infty)}\}
\]
the event where at least one block coalesces with $\{1\}$ (so $\#\Theta_1^{(\infty)} \geq 2$)
, and
\[
	\mathsf{D} = \bigcap_{n \geq 1} \{n \in \Theta_n^{(\infty)}\}
\]
the event where no coalescence occurs.
Additionally consider $\mathsf{C}_k := \mathsf{C} \cap \mathsf{S}_k$ and $\mathsf{D}_k := \mathsf{D}  \cap \mathsf{S}_k$.
Note that by construction $\mathsf{C}_1, \ldots, \mathsf{C}_d, \mathsf{D}_1, \ldots, \mathsf{D}_d$ are pairwise disjoint
and by definition of the M-$\Lambda$-MS-coalescent $\sum_{k \in [d] \setminus \{j\}} \bigl( \mathbb{P}^{\bm{\pi}(j)}( \mathsf{C}_k ) + \mathbb{P}^{\bm{\pi}(j)}( \mathsf{D}_k ) \bigr) = 1$.

Now let $k \in [d] \setminus \{j\}$ be given such that $\mathbb{P}^{\bm{\pi}(j)}(\mathsf{C}_k \cup \mathsf{D}_k) > 0$.
Consider $n \in \mathbb{N}$ given and fixed.
Let $\bm{\theta}_n$ be a $d$-type partition of $[n]$ such that $\mathbb{P}^{\bm{\pi}(j)}( \mathbf{\Theta}^{(\infty)}\vert_{[n]} = \bm{\theta}_n \mid \mathsf{C}_k ) > 0$ or such that $\mathbb{P}^{\bm{\pi}(j)}( \mathbf{\Theta}^{(\infty)}\vert_{[n]} = \bm{\theta}_n \mid \mathsf{D}_k ) > 0$.
Then we have
\begin{align}
	\label{eq:coalNK}
	\mathbb{P}^{\bm{\pi}(j)}( \mathbf{\Theta}^{(\infty)} \vert_{[n]} = \bm{\theta}_n \mid \mathsf{C}_k )
	 & = \sum_{\bm{\theta}_{n + 1} : R_n \bm{\theta}_{n + 1} = \bm{\theta}_n} \mathbb{P}^{\bm{\pi}(j)}( \mathbf{\Theta}^{(\infty)} \vert_{[n + 1]} = \bm{\theta}_{n + 1} \mid  \mathsf{C}_k )                       \\
	\shortintertext{and}
	\mathbb{P}^{\bm{\pi}(j)}( \mathbf{\Theta}^{(\infty)} \vert_{[n]} = \bm{\theta}_n \mid \mathsf{D}_k )
	 & = \sum_{\bm{\theta}_{n + 1} : R_n \bm{\theta}_{n + 1} = \bm{\theta}_n} \mathbb{P}^{\bm{\pi}(j)}( \mathbf{\Theta}^{(\infty)} \vert_{[n + 1]} = \bm{\theta}_{n + 1} \mid  \mathsf{D}_k ) \,, \label{eq:diceNK}
\end{align}
where the above sums are over the set of all $d$-type partition of $[n + 1]$, $\bm{\theta}_{n + 1}$, such that its restriction to $[n]$ coincides with $\bm{\theta}_n$.

In \eqref{eq:coalNK} we are conditioning on the event that there is a coalescence such that the first block of $\mathbf{\Theta}^{(\infty)}$ is of type $k$.
By definition of the M-$\Lambda$-MS-coalescent, the only coalescence or type transition possible is the one involving the block containing $1$.
This means that the probabilities in the sum on the right-hand side are zero if $\bm{\theta}_{n + 1}$ satisfies either of the following two conditions:
\begin{itemize}
	\item $\{n + 1\}$ as a singleton is a block of $\bm{\theta}_{n + 1}$ and its type is different from $\mathfrak{t}_{n + 1}$ (there was a type transition not involved in the coalescence);
	\item The block of $\bm{\theta}_{n + 1}$ containing $n + 1$ is not a singleton and does not contain $1$ (there was a second coalescence event).
\end{itemize}
Hence, there are only two possible values of $\bm{\theta}_{n + 1}$ for which the probabilities inside the sum on the right-hand side of \eqref{eq:coalNK} are nonzero.
Namely, either $n + 1$ merges to the block containing $1$, represented as $\bm{\theta}_{n, c}$ or $\{n + 1\}$ stays as a singleton with the same type as before, represented as $\bm{\theta}_{n, s}$.
Thus, \eqref{eq:coalNK} becomes
\begin{equation} \label{eq:lambda1}
	\mathbb{P}^{\bm{\pi}(j)}( \mathbf{\Theta}^{(\infty)}\vert_{[n]} = \bm{\theta}_n \mid \mathsf{C}_k )
	= \mathbb{P}^{\bm{\pi}(j)}( \mathbf{\Theta}^{(\infty)}\vert_{[n + 1]} = \bm{\theta}_{n, c} \mid \mathsf{C}_k ) + \mathbb{P}^{\bm{\pi}(j)}( \mathbf{\Theta}^{(\infty)}\vert_{[n + 1]} = \bm{\theta}_{n,s} \mid \mathsf{C}_k ) \,.
\end{equation}
Doing a similar analysis, the possible values of $\bm{\theta}_{n + 1}$ such that the probability in the right hand side of \eqref{eq:diceNK} is not zero correspond to $\{n + 1\}$ staying as a singleton, where its type may be any $l \in [d]$.
We represent by $\bm{\theta}_{n, l}$ the partition where the block $\{n + 1\}$ is of type $l$, with which \eqref{eq:diceNK} becomes
\begin{equation}
	\label{eq:diceProc1}
	\mathbb{P}^{\bm{\pi}(j)}( \mathbf{\Theta}^{(\infty)}\vert_{[n]} = \bm{\theta}_n \mid \mathsf{D}_k )
	= \sum_{l = 1}^d \mathbb{P}^{\bm{\pi}(j)}( \mathbf{\Theta}^{(\infty)}\vert_{[n + 1]} = \bm{\theta}_{n, l} \mid \mathsf{D}_k ) \,.
\end{equation}

We will now focus our attention on $T_{\{1, 2\}}$ to analyse the coalescence of blocks $\{1\}$ and $\{2\}$ when they start being of the same type.
Thus, for this part we will suppose that $\mathfrak{t}_2 = j$.
As before we consider the case where $T_{\{1, 2\}} < \infty$ a.s. under $\mathbb{P}^{\bm{\pi}(jj)}$.
Additionally, because the coalescence from one type to another is already covered in the previous case, we will focus on the set $\{T_{\{1, 2\}} < T_1\}$.
Similarly to before, for $n \in \mathbb{N}$ with $n \geq 2$, let $F_n$ be the event where the first jump observed by $\mathbf{\Pi}\vert_{[n]}$ is at time $T_{\{1, 2\}}$, so it involves a coalescence of the first two blocks (which are of the same type) in which the the type of block $1$ does not change.
In this case we have that $\mathbb{P}^{\bm{\pi}(jj)}(F_n) > 0$ by the Markovian consistency.
As before, construct a random $d$-type partition $\mathbf{\Xi}^{(\infty)}$ such that $\mathbf{\Xi}^{(\infty)}\vert_{[n]}$ has the same distribution as $\mathbf{\Pi}\vert_{[n]}(T_{\{1, 2\}})$ given $F_n$.
In this case, it is clear that $\mathbb{P}^{\bm{\pi}(jj)}(\mathsf{C}_j) = 1$ and furthermore, we obtain the analogue of \eqref{eq:lambda1},
\begin{equation}
	\label{eq:lambda2}
	\mathbb{P}^{\bm{\pi}(jj)}(\mathbf{\Xi}^{(\infty)} \vert_{[n]} = \bm{\theta}_{n})
	= \mathbb{P}^{\bm{\pi}(jj)}(\mathbf{\Xi}^{(\infty)} \vert_{[n + 1]} = \bm{\theta}_{n, c}) +
	\mathbb{P}^{\bm{\pi}(jj)}(\mathbf{\Xi}^{(\infty)} \vert_{[n]} = \bm{\theta}_{n, s}) \,,
\end{equation}
for $n \geq 2$.

With these preparations, we can now express the rates.
Under $\mathbb{P}^{\bm{\pi}(j)}$, $T_1$ is exponential, say with rate $\eta_1^{(j)}$, while under $\mathbb{P}^{\bm{\pi}(jj)}$, $T_1 \wedge T_{\{1, 2\}}$ is exponential, say with rate $\eta_{\{1,2\}}^{(j)}$.
Now, for a $d$-type partition $\bm{\theta}_n$ of $[n]$ write
\begin{align}
	\lambda_{n}(\bm{\pi}, \bm{\theta}_n)
	 & = \sum_{k \in [d] \setminus \{\mathfrak{t}_1\}} \eta_1^{(\mathfrak{t}_1)}
	\mathbb{P}^{\bm{\pi}(\mathfrak{t}_1)}(\{\mathbf{\Theta}^{(\infty)}\vert_{[n]} = \bm{\theta}_n\} \cap \mathsf{C}_k) \label{eq:lrateCoal}                                                                                                                                     \\
	 & \qquad + 1_{\{\mathfrak{t}_1 = \mathfrak{t}_2\}} \eta_{\{1,2\}}^{(\mathfrak{t}_1)} \mathbb{P}^{\bm{\pi}(\mathfrak{t}_1 \mathfrak{t}_1)}(T_{\{1, 2\}} < T_1) \mathbb{P}^{\bm{\pi}(\mathfrak{t}_1 \mathfrak{t}_1)}(\mathbf{\Xi}^{(\infty)}\vert_{[n]} = \bm{\theta}_n) \,, \\
	\shortintertext{and
	}
	\gamma_{n}(\bm{\pi}, \bm{\theta}_n)
	 & = \sum_{k \in [d] \setminus \{\mathfrak{t}_1\}} \eta_1^{(\mathfrak{t}_1)} \mathbb{P}^{\bm{\pi}(\mathfrak{t}_1)}(\{\mathbf{\Theta}^{(\infty)}\vert_{[n]} = \bm{\theta}_n\} \cap \mathsf{D}_k) \,. \label{eq:lrateDice}
\end{align}
Using \eqref{eq:lambda1}, \eqref{eq:lambda2}, and \eqref{eq:lrateCoal} we obtain the consistency relation
\begin{equation} \label{eq:lambdaRatesCons}
	\lambda_{n}(\bm{\pi}, \bm{\theta}_n)
	= \lambda_{n+1}(\bm{\pi}, \bm{\theta}_{n, c}) + \lambda_{n+1}(\bm{\pi}, \bm{\theta}_{n, s}) \,.
\end{equation}
On the other hand \eqref{eq:diceProc1} and \eqref{eq:lrateDice} allow us to obtain
\begin{equation} \label{eq:diceRatesCons}
	\gamma_{n}(\bm{\pi}, \bm{\theta}_n)
	= \sum_{l = 1}^d \gamma_{n + 1}(\bm{\pi}, \bm{\theta}_{n, l}) \,.
\end{equation}

Now fix $n \in \mathbb{N}$.
On one hand, $\lambda_{n}(\bm{\pi}, \bm{\theta}_n)$ is the rate at which the process $\mathbf{\Pi}\vert_{[n]}$ goes from $\bm{\pi}\vert_{[n]}$ to $\bm{\theta}_n$ in a coalescence event involving block 1 (and block 2 when the type of block 1 does not change).
On the other hand, $\gamma_{n}(\bm{\pi}, \bm{\theta}_n)$ is the rate at which at which the process $\mathbf{\Pi}\vert_{[n]}$ goes from $\bm{\pi}\vert_{[n]}$ to $\bm{\theta}_n$ in a type transition event  involving block 1.
To finish the proof,
we need to extend this to deduce the rates at which the process $\mathbf{\Pi}\vert_{[n]}$ jumps from a typed partition $\bm{\pi}\vert_{[n]}$ to another typed partition $\bm{\theta}_n$ without necessarily involving block 1.
In order to do this we will use the partial exchangeability of the M-$\Lambda$-MS-coalescent.
Suppose for a moment that $x \in [d]^\infty$ describes the initial types of the process $\mathbf{\Pi}$, so that $x_m$ is the type of the $m$-th block $\Pi_m(0) = \{m\}$ for $m \in \mathbb{N}$.
For any fixed finite permutation $\sigma : \mathbb{N} \to \mathbb{N}$ define the process $\sigma \mathbf{\Pi} := (\sigma \mathbf{\Pi}(t))_{t \geq 0}$, where for each $t \geq 0$ we put
\[
	\sigma \mathbf{\Pi}(t) := \{(\sigma \Pi_i(t), Y_i(t)) ; i \in \mathbb{N}\} \,,
\]
with $\sigma \Pi_i(t) := \{\sigma(j) : j \in \Pi_i(t)\}$ and $Y_i(t)$ being the type of block $\sigma \Pi_i(t)$ (at time $t$).
By considering a similar argument, invoking the partial exchangeability, to that given in the first part of the proof of Theorem \ref{thm:pExDP} we observe that $\sigma \mathbf{\Pi}$ has the same distribution as $\mathbf{\Pi}$ started from the typed partition $\{(\{m\}, x_{\sigma^{-1}(m)}) : m \in \mathbb{N}\}$.
By definition $\sigma^{-1}$ is still a finite permutation, so in particular we deduce that, for any permutation $\sigma : [n] \to [n]$,
\begin{equation} \label{eq:aux_pmlcProof}
	\lambda_{n}(\bm{\pi}, \bm{\theta}_n) = \lambda_{n}(\sigma\bm{\pi}, \sigma\bm{\theta}_n) \quad\text{and}\quad
	\gamma_{n}(\bm{\pi}, \bm{\theta}_n) = \gamma_{n}(\sigma\bm{\pi}, \sigma\bm{\theta}_n)\,,
\end{equation}
where we put
\[
	\sigma \bm{\pi} = \{(\sigma(\pi_i), \mathfrak{t}_i) : i \in \mathbb{N}\} \quad\text{with}\quad \sigma(\pi_i) = \{\sigma(j) : j \in \pi_i\} \,,
\]
and $\sigma\bm{\theta}_n$ is defined similarly.
This gives us the same rates $\lambda_n,\gamma_n$ as before without requiring that block 1 is involved in a transition.

Expression \eqref{eq:aux_pmlcProof} tells us that, for any given $n \in \mathbb{N}$, the rate at which $\mathbf{\Pi}\vert_{[n]}$ goes from $\bm{\pi}\vert_{[n]}$ to $\bm{\theta}_n$ only depends on:
\begin{enumerate}
	\item The number of blocks of each type $i \in [d]$ that coalesce into a block of type $k \in [d]$ in a coalescent event.
	\item The number of blocks of each type $i$ that move to a block of type $k \in [d]$ for each $i, k \in [d]$ in a type transition event.
\end{enumerate}
Therefore, in the coalescent events, Theorems 1.5 and 1.7 in \cite{johnstonMultitypeLcoalescents2023} tell us that \eqref{eq:lambdaRatesCons} implies that $\mathbf{\Pi}$ has a coalescent mechanism given by an M-$\Lambda$-coalescent.
Indeed, the first equality in \eqref{eq:aux_pmlcProof} shows that, for each $j \in [d]$, the rates of coalescent events producing a new block of type $j$ depend only on the current number of blocks $(b_1, \ldots, b_d)$ of each type and on the number of blocks $(k_1, \ldots, k_d)$ of each type participating in the merger.
That is, following the notation in \cite{johnstonMultitypeLcoalescents2023}, for each $j \in [d]$, there exists an array of nonnegative numbers $(\lambda_{\mathbf{b}, \mathbf{k} \to j} : \mathbf{b} \in \mathbb{N}_0^d \setminus \{0, e_j\}, \mathbf{k} \in [\mathbf{b}]_0 \setminus \{0\})$, where, for $\mathbf{b} = (b_1, \ldots, b_d)$ and $\mathbf{k} = (k_1, \ldots, k_d)$, $\lambda_{\mathbf{b}, \mathbf{k} \to j}$ is the rate corresponding to the event in which $k_1$ blocks of type $1$, $k_2$ blocks of type $2$, etc., merge into a single block of type $j$, given there are $b_1$ blocks of type $1$, $b_2$ blocks of type $2$, etc.
Moreover, by \eqref{eq:lambdaRatesCons}, this array satisfies the consistency relation
\[
	\lambda_{\mathbf{b}, \mathbf{k} \to j} = \lambda_{\mathbf{b} + e_l, \mathbf{k} + e_l \to j} + \lambda_{\mathbf{b} + e_l, \mathbf{k} \to j}
\]
for each $l \in [d]$, which is precisely equation (7) in \cite{johnstonMultitypeLcoalescents2023}.
By Theorem 1.7 of \cite{johnstonMultitypeLcoalescents2023}, there exist nonnegative constants $\rho_{jj \to j}$, $\rho_{l \to j}$ for $l \neq j$, and a measure $\mathcal{Q}_{\to j}$ on $[0, 1]^d$ such that $\int_{[0, 1]^d} \sum_{l = 1}^d s_l^{1 + \delta_{jl}} \mathcal{Q}_{\to j} < \infty$, and
\[
	\lambda_{\mathbf{b}, \mathbf{k} \to j} = \sum_{l \in [d] \setminus \{j\}} \rho_{l \to j} 1_{\{\mathbf{k} = e_l\}} + \rho_{jj \to j} 1_{\{\mathbf{k} = 2 e_j\}}
	+ \int_{[0, 1]^d} \prod_{l = 1}^d s_l^{k_l} (1 - s_l) ^{b_l - k_l} \mathcal{Q}_{\to j}(ds) \,.
\]
This is precisely the definition of the rates of the M-$\Lambda$-coalescent given in \cite{johnstonMultitypeLcoalescents2023} (cf. \eqref{eq:jkrRates}), meaning that the mechanism controlling the coalescent part of $\mathbf{\Pi}$ is precisely a M-$\Lambda$-coalescent.

Meanwhile, for the type transition events,
the second equation in \eqref{eq:aux_pmlcProof} show that the rates of type switching are governed by a mechanism that involves only $(b, K)$-changes in the sense described in Section \ref{subsect:particle_interpretation}.
That is, there is an array $(\gamma_{b, K}, b \in \mathbb{N}_0^d \setminus \{0\}, K \in \mathcal{S}_{d-1,d}(b) \setminus \{\diag(b)\})$ where $\gamma_{b, K}$ denotes the rate of a type switching event starting from a configuration with $b_1$ blocks of type $1$, $b_2$ blocks of type $2$, etc., in which $k_{lj}$ blocks of type $l$ switch to type $j$ for each $l, j \in [d]$.
By \eqref{eq:diceRatesCons}, these rates satisfy the consistency equation \eqref{eq:consistencyEq}.
Hence, by Lemma \ref{lem:bkRatesExplicit} together with Proposition \ref{th:cerwCharact}, they coincide with the rates of a dice process, and therefore the type switching mechanism is itself described by a dice process.

Finally, since by construction every coalescent event necessarily involves the merger of at least two blocks (cf. the events $\mathsf{C}_1, \ldots, \mathsf{C}_d$), it follows that $\rho_{l \to j} = 0$ for all $l \neq j$.
Indeed, such terms would otherwise correspond to type switching events where a single block changes its type, and these events are instead governed by the type switching mechanism.
Moreover, the coalescent and type switching mechanisms are independent, as a coalescence event and a type switching event cannot occur simultaneously.
These observations are crucial for ensuring the uniqueness of the decomposition of an M-$\Lambda$-MS-coalescent into two independent components: an M-$\Lambda$-coalescent governing mergers, and a Dice process governing type switching.
Without this restriction, infinitely many decompositions would be possible.
Thus, the characterization is complete.
\section*{Acknowledgements}
J. L. P\'erez gratefully acknowledges the support of the Fulbright Program during his sabbatical stay at Arizona State University. He also wishes to thank the faculty and staff of the School of Mathematical and Statistical Sciences at ASU for their hospitality.
I. Nu\~nez acknowledges the support of the UNAM-PAPIIT grant IN101722.
\appendix

\section{Proof of the de Finetti process and duality}
\label{Prop_2.2}

In this section, we provide the proof of  Proposition \ref{prop:convergence_deFinetti}. We divide the proof in two steps: first we prove that there exists a unique solution to \eqref{eq:deFinettiSDE} and then we prove the convergence of the sequence of processes $(R^{(n)})_{n \geq 1}$ in the space $D([0, T], \Delta_{d-1})$.

\paragraph{Existence and uniqueness of a solution to the SDE}
(i) We first show the uniqueness of solutions to the SDE given in \eqref{eq:deFinettiSDE}. To this end, let us consider two solutions to \eqref{eq:deFinettiSDE}, denoted by $(R(t))_{t\geq0}$ and $(\overline{R}(t))_{t\geq0}$.
For $t \geq 0$ put $\zeta(t) := R(t) - \overline{R}(t)$.
Then, for each $i \in [d]$,
\begin{align}\label{eq:deFinettiSDEIneq}
	\abs{ \zeta_i(t) }
	 & \leq \abs{\zeta_i(0)} + \sum_{j \in [d] \setminus \{i\}} \int_0^t \bigl[ a_{ji} \abs{\zeta_j(s)} + a_{ij} \abs{\zeta_i(s)} \bigr] ds                \\
	 & \quad + \int_0^t \int_{\Delta_{d-1}^d} \sum_{j \in [d] \setminus \{i\}} \bigl[ u_{ji} \abs{\zeta_j(s-)} + u_{ij} \abs{\zeta_i(s-)} \bigr] N(ds, dU)
\end{align}
Summing over $i \in [d]$ and taking expectations yields
\begin{equation} \label{aux_sol_1}
	\sum_{i \in [d]} \mathbb{E}\Bigl[ \sup_{0 \leq s \leq t} \abs{\zeta_i(s)} \Bigr]
	\leq \abs{\zeta(0)} + 2 (d \tilde{a} + \tilde{I}^{\nu}) \int_0^t \sum_{i \in [d]} \mathbb{E}\Bigl[ \sup_{0 \leq s' \leq s} \abs{\zeta_i(s')} \Bigr] ds
\end{equation}
where
\begin{equation} \label{eq:auxiliarySupConstants}
	\tilde{a} = \sup_{i, j \in [d] : i \neq j} a_{ij}
	\quad\text{and}\quad
	\tilde{I}^{\nu} = \sup_{i \in [d]} \int_{\Delta_{d-1}^d} (1 - u_{ii}) \nu(dU) \,.
\end{equation}
Thus, there exists a constant $C>0$ such that
\[
	\sum_{i \in [d]} \mathbb{E}\Bigl[ \sup_{0 \leq s \leq t} \abs{\zeta_i(s)} \Bigr] \leq \abs{\zeta(0)} + C \int_0^t \sum_{i \in [d]} \mathbb{E}\Bigl[ \sup_{0 \leq s' \leq s} \abs{\zeta_i(s')} \Bigr] ds \,, \quad t \geq 0 \,.
\]
Applying Gronwall's inequality, we conclude that for all $t\geq0$,
\[
	\sum_{i \in [d]} \mathbb{E}\Bigl[ \sup_{0 \leq s \leq t} \abs{\zeta_i(s)} \Bigr] \leq \abs{\zeta(0)} e^{Ct} \,.
\]
In particular, if $\zeta(0) = 0$, or equivalently $R(0) = \overline{R}(0)$,
\[
	\sum_{i \in [d]} \mathbb{E}\Bigl[ \sup_{0 \leq s \leq t} \abs{ R_i(s) - \overline{R}_i(s) } \Bigr] = 0 \,.
\]
which implies, by the right-continuity of $R$ and $\overline{R}$, that
$\mathbb{P}\left(R(t) = \overline{R}(t)\text{ for all $t\geq0$}\right)=1$.

(ii) We now address the existence of a solution to \eqref{eq:deFinettiSDE}.
For each $n\geq1$, define the set $V_d^n = \bigcup_{i \in [d]} \{U \in \Delta_{d-1}^d : u_{ii} < 1-1/n\}$. Note that by \eqref{eq:nuConstraint},
\[\nu(V_d^n) \leq \sum_{i \in [d]} \nu(1/n < 1 - u_{ii}) \leq n \int_{\Delta_{d-1}^d} \sum_{i \in [d]} (1 - u_{ii}) \nu(dU) < \infty \,. \]
Thus, following the construction in Proposition 2.2 of \cite{FU2010306}, we can build a pathwise unique solution $(R^n(t))_{t\geq0}$ to the following SDE:
\begin{equation} \label{aux_sol}
	dR_i^n(t) = \sum_{j \in [d] \setminus \{i\}} \bigl( a_{ji} R_j^n(t) - a_{ij} R_i^n(t) \bigr) dt
	+ \int_{V_d^n} \sum_{j \in [d] \setminus \{i\}} \bigl( u_{ji} R_{j}^n(t-) - u_{ij} R_i^n(t-) \bigr) N(dt, dU) \,,
\end{equation}
with initial condition $R^n(0)=(r_1,\dots,r_d)$.

Proceeding similarly to \eqref{aux_sol_1}, we obtain
\begin{align*}
	\sum_{i \in [d]} \mathbb{E}\Bigl[ \sup_{0 \leq s \leq t} \abs{ R_i^n(s) } \Bigr]
	 & \leq 1 + 2 (d \tilde{a} + \tilde{I}^{\nu}) \int_0^t \sum_{i \in [d]} \mathbb{E}\Bigl[ \sup_{0 \leq s' \leq s} \abs{R_i^n(s')} \Bigr] ds \\
	 & = 1 + C \int_0^t \sum_{i \in [d]} \mathbb{E}\Bigl[ \sup_{0 \leq s' \leq s} \abs{R_i^n(s')} \Bigr] ds \,,
\end{align*}
with $\tilde{a}$ and $\tilde{I}^{\nu}$ being the constants defined in \eqref{eq:auxiliarySupConstants}.
By Gronwall's inequality, it follows that
\[
	\sum_{i\in[d]}\E\left[ \sup_{0\leq s\leq t}|R_i^n(s)| \right]\leq e^{Ct}.
\]
Fix $n\leq m$ and $T>0$. Then, proceeding as before, for $t\leq T$,
\begin{align*}
	\sum_{i \in [d]} \mathbb{E}\Bigl[ \sup_{0 \leq s \leq t} \abs{ R_i^n(s) - R_i^m(s) } \Bigr]
	 & \leq 2 ( d \tilde{a} + \tilde{I}^{\nu}_n ) \int_0^t \sum_{i \in [d]} \mathbb{E}\Bigl[ \sup_{0 \leq s' \leq s} \abs{R_i^n(s') - R_i^m(s')} \Bigr] ds \\
	 & \quad + 2 \tilde{I}^{\nu}_{n, m} \int_0^t \mathbb{E}\Bigl[ \sup_{0 \leq s' \leq s} \abs{R_i^m(s')} \Bigr] ds                                        \\
	 & \leq C(n, m, T) + C \int_0^t \sum_{i \in [d]} \mathbb{E}\Bigl[ \sup_{0 \leq s' \leq s} \abs{R_i^n(s') - R_i^m(s')} \Bigr] ds \,,
\end{align*}
where
\[
	\tilde{I}^{\nu}_n = \sup_{i \in [d]} \int_{V_d^n} (1 - u_{ii}) \nu(dU)
	\quad \text{and}\quad
	\tilde{I}^{\nu}_{n, m} \sup_{i \in [d]} \int_{V_d^m \setminus V_d^n} (1 - u_{ii}) \nu(dU) \,,
\]
and
\[
	C(n, m , T) := 2 \tilde{I}^{\nu}_{n, m} \int_0^t e^{Cs} ds \,.
\]
Applying Gronwall's inequality again, we deduce that for $\lambda>0$ and any $t \leq T$,
\begin{align*}
	\lim_{n,m\to\infty}\mathbb{P}\left(\sum_{i\in[d]}\sup_{0\leq s\leq t}\left|R_i^n(s)-R^m_i(s)\right|\geq\lambda\right) & \leq \lim_{n,m\to\infty}\frac{1}{\lambda}
	\sum_{i\in[d]}\E\left[\sup_{0\leq s\leq t}\left|R_i^n(s)-R^m_i(s)\right|\right]\notag                                                                             \\&\leq \lim_{n,m\to\infty}C(n,m,T)e^{CT}=0,
\end{align*}
where the last equality follows from the fact that
\begin{equation}\label{uni_con}
	\int_{\Delta_{d-1}^d}\left(\sum_{j \in [d] \setminus \{i\}}u_{ij}\right)\nu(dU)=\int_{\Delta_{d-1}^d}\left(1-u_{ii}\right)\nu(dU)<\infty.
\end{equation}
Then, we can find a subsequence $(n_k)_{k\geq1}$ such that for any $k \in\mathbb{N}$,
\begin{align*}
	\mathbb{P}\left(\sum_{i\in[d]}\sup_{0\leq s\leq t}\left|R_i^{n_{k+1}}(s)-R^{n_k}_i(s)\right|\geq2^{-k}\right)<2^{-k}.
\end{align*}
Thus, by Borel--Cantelli's Lemma there exists a measurable set $N$ such that $\mathbb{P}(N)=0$, and for $\omega\not\in N$ there exists $K(\omega)\in\mathbb{N}$ such that if $k\geq K(\omega)$ then
\begin{align*}
	\sum_{i\in[d]}\sup_{0\leq s\leq t}\left|R_i^{n_{k+1}}(s)-R^{n_k}_i(s)\right|(\omega)\leq2^{-k}.
\end{align*}
Hence $(R^n(\omega))_{n\geq1}$ is a Cauchy sequence in the space of càdlàg functions from $[0,T]$ to $\mathbb{R}$ equipped with the supremum norm.
Hence, there exists $R(\omega)\in\mathbb{D}([0,T])$ such that
\begin{align}\label{uni_conv}
	\lim_{k\to\infty}\sup_{0\leq s\leq T}|R^{n_k}(s)-R(s)|(\omega)=0,\qquad \omega\not\in N.
\end{align}
Hence,  $(R^n(\omega))_{n\geq1}$ converges a.s. to the process $R$ in the space of càdlàg functions from $[0,T]$ to $\mathbb{R}$ equipped with the supremum norm.
Passing to the limit in \eqref{aux_sol} along this subsequence, and using \eqref{uni_con}, we conclude that $R$ satisfies \eqref{eq:deFinettiSDE}.

\paragraph{Proof of the convergence in \texorpdfstring{$D([0, T], \Delta_{d-1})$}{D([0,T],Dd-1)}}

For this part we need some notation.
For any $\bm{r} \in \Delta_{d-1}$ define $\bm{r}^{(n)}$ as the element in $n^{-1} \mathcal{S}(n)$ whose distance to $\bm{r}$ is minimal so $\abs{\bm{r} - \bm{r}^{(n)}} \leq 1/n$.
In case of ties, choose the minimal element with respect to the lexicographical order.
Given this notation, let us make a short heuristic argument which we will formalize afterwards.
If for each $n \in \mathbb{N}$ and each $i \in [d]$ we define
\[
	R_i^{(n)}(t) := \frac{ \#\{ j \in [n] : X_j^{(\infty)}(t) = i\} }{n} \,, t \geq 0 \,.
\]
Then $R^{(n)} = (R_1^{(n)}, \ldots, R_d^{(n)})$  is a Markov process on the set $n^{-1} \mathcal{S}_{d-1}(n)$ whose generator is given by
\begin{align*}
	\mathcal{A}^n f(\bm{r}) = {}
	 & \sum_{i \in [d]} \sum_{j \in [d] \setminus \{i\}} a_{ij} n r_i  \Bigl( f\Bigl( \bm{r} + \frac{1}{n} (e_j - e_i) \Bigr) - f(\bm{r}) \Bigr) \\
	 & + \int_{\Delta_{d-1}^d} \mathbb{E}_{n\bm{r}, U}\Bigl[ f\Bigl(\frac{1}{n} \sum_{i = 1}^d Y^{(i)}\Bigr) - f(\bm{r}) \Bigr] \nu(dU) \,,
\end{align*}
where under $\mathbb{E}_{n \bm{r}, U}$, $Y^{(i)}$ has multinomial distribution with parameters $n r_i$ and $(u_{i1}, \ldots, u_{id})$.
The convergence is obtained by noting that taking $n$ sufficiently large, for $\bm{r} \in \Delta_{d-1}$,
\[
	n \Bigl( f\bigl( \bm{r}^{(n)} + \frac{1}{n} (e_j - e_i) \bigr) - f(\bm{r}^{(n)}) \Bigr) \approx (\partial_j - \partial_i) f(\bm{r}),
\]
and that for every $i \in [d]$,
\[
	\mathbb{E}_{n \bm{r}^{(n)}, U} \Bigl[ f \Bigl( \frac{1}{n} \sum_{i = 1}^d Y^{(i)} \Bigr) - f(\bm{r^{(n)}}) \Bigr]
	\approx f(U^T \bm{r}) - f(\bm{r}) \,,
\]
by a law of large numbers.
We now proceed with the formal argument.

Fix $\bm{r} \in \Delta_{d-1}$.
Let us consider the Markov process $R^{(n)} = (R_1^{(n)}, \ldots, R_d^{(n)})$ with initial state $\bm{r}^{(n)}$ for $\bm{r} \in \Delta_{d-1}$.
We can extend the generator of the process $R^{(n)}$ to $\Delta_{d-1}$ by setting
\begin{align}\label{con_1}
	\mathcal{A}^n f(\bm{r}) = {}
	 & \sum_{i \in [d]} \sum_{j \in [d] \setminus \{i\}} a_{ij} n r_i^{(n)}  \left[ f\left( \bm{r}^{(n)} + \frac{1}{n} (e_j - e_i) \right) - f\left(\bm{r}^{(n)}\right) \right] \\
	 & + \int_{\Delta_{d-1}^d} \mathbb{E}_{n\bm{r}^{(n)}, U}\left[ f\Bigl(\frac{1}{n} \sum_{i = 1}^d Y^{(i)}\Bigr) - f\left(\bm{r}^{(n)}\right) \right] \nu(dU) \,,
\end{align}
where under $\mathbb{E}_{n \bm{r}^{(n)}, U}$, $Y^{(i)}$ has multinomial distribution with parameters $n r_i^{(n)}$ and $(u_{i1}, \ldots, u_{id})$.

By Taylor's theorem for any $f\in\mathcal{C}^2([0,1]^d)$ we have that there exists a constant $C(f)>0$, depending only on $f$, such that
\begin{align}\label{con_2}
	B_1^n(i,j) & := \Bigg|f\left( \bm{r}^{(n)} + \frac{1}{n} (e_j - e_i) \right) - f\left(\bm{r}^{(n)}\right)-f\left( \bm{r} + \frac{1}{n} (e_j - e_i) \right) + f\left(\bm{r}\right)\Bigg| \notag \\
	           & \leq C(f)\left(\left|r_i^{(n)}-r_i\right|^2+\left|r_j^{(n)}-r_j\right|^2+2\left|r_i^{(n)}-r_i\right|\left|r_j^{(n)}-r_j\right|\right)\notag                                       \\
	           & \leq \frac{C_1(f)}{n^2},
\end{align}
where $C_1(f)>0$ is a constant depending only of $f$.
Using a similar argument we can find, by Taylor's Theorem, a constant $C_2(f)>0$ such that
\begin{align}\label{con_3}
	B_2^n(i,j):=\Bigg|f\left( \bm{r} + \frac{1}{n} (e_j - e_i) \right) + f\left(\bm{r}\right)-\frac{1}{n}(\partial_j - \partial_i) f(\bm{r})\Bigg|\leq \frac{C_2(f)}{n^2}.
\end{align}
The fact that $\bm{r}\in\Delta_{d-1}$ implies that there exists a constant $C_4(f)$ such that
\begin{align}\label{con_4}
	B_3^n(i,j) & :=\Bigg|\left(r_i^{(n)}-r_i\right)  \Bigg[ f\left( \bm{r}^{(n)} + \frac{1}{n} (e_j - e_i) \right) - f\left(\bm{r}^{(n)}\right) \Bigg]\Bigg| \notag \\
	           & \leq C_4(f)\bigl|r_i^{(n)}-r_i\bigr|\left(\left|r_i^{(n)}-r_i\right|+\left|r_j^{(n)}-r_j\right|\right)\notag                                       \\
	           & \leq 2\frac{C_4(f)}{n^2}.
\end{align}
Hence, using \eqref{con_2}, \eqref{con_3}, and \eqref{con_4} gives
\begin{align}\label{con_5} \MoveEqLeft
	\Bigg|\sum_{i \in [d]} \sum_{j \in [d] \setminus \{i\}} a_{ij} n r_i^{(n)}  \left[ f\left( \bm{r}^{(n)} + \frac{1}{n} (e_j - e_i) \right) - f\left(\bm{r}^{(n)}\right) \right]-\sum_{i \in [d]} \sum_{j \in [d] \setminus \{i\}} a_{ij}r_i(\partial_j - \partial_i) f(\bm{r})\Bigg|\notag \\
	 & \leq \sum_{i \in [d]} \sum_{j \in [d] \setminus \{i\}} a_{ij}n(B_1^n(i,j)+B_2^n(i,j)+B_3^n(i,j))\notag                                                                                                                                                                                 \\
	 & \leq  \frac{C_5(f)}{n}\sum_{i \in [d]} \sum_{j \in [d] \setminus \{i\}} a_{ij},
\end{align}
for some constant $C_5(f)>0$ that only depends on $f$.

Now for the integral term in \eqref{con_1}, we have by Taylor's Theorem,
\begin{align*}
	\mathbb{E}_{n\bm{r}^{(n)}, U} & \left[ f\Bigl(\frac{1}{n} \sum_{i = 1}^d Y^{(i)}\Bigr) - f\left(U^T \bm{r}^{(n)}\right) \right] = \sum_{j=1}^d \partial_j f\left(U^T \bm{r}^{(n)}\right)\mathbb{E}_{n\bm{r}^{(n)}, U}\left[ \frac{1}{n} \sum_{i = 1}^d Y^{(i)} - U^T \bm{r}^{(n)} \right]\notag \\
	                              & +\sum_{|\beta|=2}R_\beta\left(U^T \bm{r}^{(n)}\right)\mathbb{E}_{n\bm{r}^{(n)}, U}\left[\left(\frac{1}{n} \sum_{i = 1}^d Y^{(i)}-U^T \bm{r}^{(n)}\right)^\beta\right],
\end{align*}
where $\beta \in\{0,1,2\}^2$ and $R_\beta$ satisfies that there exists a constant $C_{\beta}>0$ such that
\[
	R_\beta\left(\bm{x}\right)<C_{\beta},\qquad \text{for $\bm{x}\in\mathbb{R}^d$ such that $|x_i|\leq 1$ for $i=1,\dots,d$.}
\]
The fact that $Y^{(i)}$ has multinomial distribution with parameters $r_i^{(n)}$ and $(u_{i1}, \ldots, u_{id})$, implies that
\[
	\mathbb{E}_{n\bm{r}^{(n)}, U}\left[ \frac{1}{n} \sum_{i = 1}^d Y^{(i)}-U^T\bm{r}^{(n)} \right]=0.
\]
Thus, we can find a constant $K>0$ such that
\begin{align}\label{con_6}
	\Bigg|\mathbb{E}_{n\bm{r}^{(n)}, U} & \left[ f\Bigl(\frac{1}{n} \sum_{i = 1}^d Y^{(i)}\Bigr) - f\left(U^T\bm{r}^{(n)}\right) \right]\Bigg|
	\leq K\sum_{j=1}^d\mathbb{E}_{n\bm{r}^{(n)}, U}\left[ \left(\frac{1}{n} \sum_{i = 1}^d Y^{(i)}- U^T \bm{r}^{(n)}\right)_j^2 \right] \notag                                                                                                 \\
	                                    & +K\sum_{k\not=j}\mathbb{E}_{n\bm{r}^{(n)}, U}\left[ \left(\frac{1}{n} \sum_{i = 1}^d Y^{(i)}-U^T\bm{r}^{(n)}\right)_j\left(\frac{1}{n} \sum_{i = 1}^d Y^{(i)}-U^T\bm{r}^{(n)}\right)_k \right]\notag \\
	                                    & \leq \frac{K}{n^2}\left[\sum_{i=1}^d\sum_{j=1}^d u_{ij}(1-u_{ij}) r_i^{(n)} + \sum_{i=1}^d\sum_{k\not=j}u_{ij}u_{ik} r_i^{(n)}\right]\notag                                                          \\
	                                    & \leq 2\frac{K}{n}\sum_{i=1}^d\sum_{j=1}^d u_{ij}(1-u_{ij}).
\end{align}
On the other hand, proceeding like in \eqref{con_2},
\begin{align}\label{con_7}
	\Bigg|f\left(U^T \bm{r}^{(n)}\right) & -f\left(\bm{r}^{(n)}\right)-f(U^T\bm{r})+f(\bm{r})\Bigg|\leq D(f)\sum_{i=1}^d\left|\sum_{j=1}^d u_{ij} r_j^{(n)} - r_i^{(n)} -\sum_{j=1}^du_{ij}\bm{r_j}+\bm{r_i}\right|\notag \\
	                                     & \leq D(f)\sum_{i=1}^d\left|\sum_{j\not=i}u_{ij}\left(r_j^{(n)}-r_j\right)+(u_{ii}-1)\left(r_i^{(n)}-r_i\right)\right|\notag                                                    \\&\leq\frac{2D(f)}{n}\sum_{i=1}^d(1-u_{ii}),
\end{align}
where $D(f)>0$ is a constant dependent only on $f$.

Hence, applying \eqref{con_6} and \eqref{con_7}, gives
\begin{align}\label{con_8}
	\int_{\Delta_{d-1}^d}\Bigg| \mathbb{E}_{n \bm{r}^{(n)}, U} & \left[ f\Bigl(\frac{1}{n} \sum_{i = 1}^d Y^{(i)}\Bigr) - f\left(\bm{r}^{(n)}\right) \right]-f(U^T\bm{r})+f(\bm{r})\Bigg| \nu(dU)\notag \\&\leq \frac{K(f)}{n}\sum_{i=1}^d\int_{\Delta_{d-1}^d}(1-u_{ii})\nu(dU),
\end{align}
where the integral in the last inequality is finite by \eqref{eq:nuConstraint}.

Finally, using \eqref{con_5} together \eqref{con_8}, we obtain for every $\bm{r}\in\Delta_{d-1}$,
\begin{align*}
	\lim_{n\to\infty}\left|\mathcal{A}^nf(\bm{r})-\mathcal{A}f(\bm{r})\right|=0.
\end{align*}

\paragraph{Proof of Proposition \ref{prop:duality}}
By the assumption we made on $A$ and $\nu,$ we have that if $f_n(r) = r^n$ for $r \in \Delta_{d-1}$ and a given $n \in \mathbb{N}_0^d$, then
\begin{align*} \MoveEqLeft
	\sum_{i \in [d]} \sum_{j \in [d] \setminus \{i\}} n_i (a_{ji} r^{n+e_j-e_i} - a_{ij} r^n)   \\
	 & = \sum_{i \in [d]} \sum_{j \in [d] \setminus \{i\}} n_i a_{ji} (r^{n + e_j - e_i} - r^n)
	+ \sum_{i \in [d]} \sum_{j \in [d] \setminus \{i\}} n_i (a_{ji} - a_{ij}) r^n
\end{align*}
Now,
\[
	\sum_{i \in [d]} \sum_{j \in [d] \setminus \{i\}} n_i (a_{ji} - a_{ij}) r^n = 0
\]
for all values of $r$ and $n$ if and only if $\sum_{j \in [d] \setminus \{i\}} (a_{ij} - a_{ji}) = 0$ for all $i \in [d]$.

On the other hand, observe that for a given $U \in \Delta_{d - 1}^d$ and $n \in \mathbb{N}_0^d$,
\begin{align}\label{md_1}
	(U^T r)^n
	 & = \prod_{i = 1}^d \Bigl( \sum_{j = 1}^d u_{ji} r_j \Bigr)^{n_i}
	= \prod_{i = 1}^d \sum_{k_i \in \mathcal{S}_{d-1}(n_i)} \binom{n_i}{k_i} \prod_{j = 1}^d u_{ji}^{k_{ij}} r_j^{k_{ij}}                                                                     \\
	 & = \sum_{K \in \mathcal{S}_{d-1,d}(n)} \Bigl( \prod_{i = 1}^d \binom{n_i}{k_i} \prod_{j = 1}^d u_{ji}^{k_{ij}} \Bigr) \prod_{j = 1}^d r_{j}^{\sum_{i = 1}^d k_{ij}}                     \\
	 & = \sum_{K \in \mathcal{S}_{d-1,d}(n)} \Bigl( \prod_{i = 1}^d \binom{n_i}{k_i} \prod_{j = 1}^d u_{ji}^{k_{ij}} \Bigr) \Bigl( \prod_{j = 1}^d r_{j}^{\sum_{i = 1}^d k_{ij}} - r^n \Bigr)
	+ r^n \prod_{i = 1}^d \Bigl( \sum_{j = 1}^d u_{ji} \Bigr)^{n_i} \,,
\end{align}
where we recall that $K = (k_1, \ldots, k_d)$ and $k_i = (k_{i1}, \ldots, k_{i d})$.

Hence,
\begin{align}\label{md_2}
	(U^T r)^n - r^n
	= \sum_{K \in \mathcal{S}_{d-1}(n)} \Bigl( \prod_{i = 1}^d \binom{n_i}{k_i} \prod_{j = 1}^d u_{ji}^{k_{ij}} \Bigr) \Bigl( \prod_{j = 1}^d r_{j}^{\sum_{i = 1}^d k_{ij}} - r^n \Bigr)
	+ r^n \Bigl( \prod_{i = 1}^d \Bigl( \sum_{j = 1}^d u_{ji} \Bigr)^{n_i} - 1 \Bigr) \,.
\end{align}
Note that
\[
	r^n \Bigl( \prod_{i = 1}^d \Bigl( \sum_{j = 1}^d u_{ji} \Bigr)^{n_i} - 1 \Bigr) = 0
\]
for each $r^n$
and $n \in \mathbb{N}_0^d$ if and only if $\sum_{j = 1}^d u_{ji} = 1$.

By the previous computations we deduce that, under the assumptions we made, an application of \eqref{eq:generatorFinetti} to $f_n$ together with identities \eqref{md_1} and \eqref{md_2}, gives us
\begin{align*}
	\mathcal{A} f_n(r)
	 & = \sum_{i \in [d]} \sum_{j \in [d] \setminus \{i\}} n_i a_{ji} (r^{n + e_j - e_i} - r^n)
	+ \sum_{K \in \mathcal{S}_{d-1,d}(n)} \prod_{i = 1}^d \binom{n_i}{k_i} \widehat{\gamma}_{n, K} \Bigl( \prod_{j = 1}^d r_{j}^{\sum_{i = 1}^d k_{ij}} - r^n \Bigr) \,,
\end{align*}
where
\[
	\widehat{\gamma}_{n, K}
	= \int_{\Delta_{d-1}^d} \prod_{i = 1}^d \prod_{j = 1}^d u_{ij}^{k_{ji}} \mathbf{T}^{\mathrm{transp}}\nu(dU)
\]
and $\mathbf{T}^{\mathrm{transp}}\nu(dU)$ is the pushforward measure of $\nu$ under the mapping $U \mapsto U^T$.
Thus, we have proved that generator duality holds true.  To pass from here to the moment duality, one proves that $R$ is a Feller process, and applies Proposition 3.1 in \cite{jansenNotionDualityMarkov2014} or Theorem 3.42 in \cite{liggettContinuousTimeMarkov2010}. We leave out the details of this step. \hfill $\Box$

\section{Construction of array satisfying (\ref{eq:exchArray1})}
\label{app:PExArray}

In this appendix we will carry out the construction of an array $Z^{(ij)} = (Z_{l \ell}^{(ij)})_{l \in [d], \ell \in \mathbb{N}}$ that satisfies equation \eqref{eq:exchArray1}.
To ease notation, we drop throughout this appendix the superindexes; hence, we will write $Z_{l\ell}$, and \eqref{eq:exchArray1} becomes
\[
	\mathbb{P}'\Bigl( \bigcap_{l \in [d] : b_l > 0} \bigcap_{\ell \in [b_l]} \{Z_{l\ell} = x_{\ell}^{(l)} \} \Bigr) = \tilde{\gamma}_{b, K} \,,
\]
where $\tilde{\gamma} = (\tilde{\gamma}_{b, K} : b \in \mathbb{N}_0^d, K \in \mathcal{S}_{d-1,d}(b))$ is an array satisfying that $\tilde{\gamma}_{0, 0} = 1$ and
\[
	\tilde{\gamma}_{b, K} = \sum_{\ell \in [d]} \gamma_{b + e_l, K + E_{l\ell}}
	\quad\text{for any } l \in [d] \,.
\]

For the construction we will consider $(\Omega', \mathbb{P}') = ([0, 1), \lambda)$, where $[0, 1)$ is taken with the Borel sets and $\lambda$ is the Lebesgue measure.
Moreover, let $V$ be defined on $\Omega'$ and having uniform distribution, so $\mathbb{P}'(V \in B) = \lambda(B)$ for any Borel set $B$.
We will now construct a collection of Borel sets $\mathfrak{B} = (B_x : x \in \{0\} \cup \bigcup_{n \in \mathbb{N}} [d]^n)$ inductively on $n$.
In order to do this we will use the map $r : \mathbb{N} \to [d]$, where $r(n) = d$ if $n$ is divisible by $d$ and the remainder of $n$ when divided by $d$ otherwise.
Additionally let us define $\mathfrak{v} : \mathbb{N} \to \mathbb{N}^d$ and $\mathfrak{K} : \bigcup_{n \in \mathbb{N}} (\{n\} \times [d]^n) \to \mathbb{N}_0^{d \times d}$ respectively by
\[
	\mathfrak{v}(n) = \sum_{m = 1}^n e_{r(m)}
	\quad\text{and}\quad
	\mathfrak{K}(n, x) = \sum_{m = 1}^n E_{r(m) x_m}
\]
for $n \in \mathbb{N}$ and $x \in [d]^n$.
For a set $A \subset \mathbb{R}$ and a number $a \in \mathbb{R}$ we will use the notation $a + A$ to denote the set $\{a + b : b \in A\}$.

The construction of $\mathfrak{B}$ is as follows:
\begin{enumerate}
	\item Put $B_0 = [0, 1)$, and note that it has length $\tilde{\gamma}_{0, 0}$.
	\item For $n = 1$, put $B_i = \sum_{j = 1}^{i-1} \tilde{\gamma}_{e_1, E_{1j}} + [0 , \tilde{\gamma}_{e_1, E_{1i}})$ for $i \in [d]$, and note that $B_i$ has length $\tilde{\gamma}_{e_1, E_{1i}}$.
	\item Suppose that $B_{x}$ has been defined for $x \in \{0\} \cup \bigcup_{m = 1}^n [d]^m$ for some $n \in \mathbb{N}$.
	      For any $y \in [d]^{n + 1}$ put
	      \[
		      B_{y} := \Bigl(\inf B_{R_n y} + \sum_{j = 1}^{y_{n + 1} - 1} \tilde{\gamma}_{\mathfrak{v}(n + 1), \mathfrak{K}(n + 1, y + (j - y_{n + 1}) e_{n + 1})} \Bigr) + [0, \tilde{\gamma}_{\mathfrak{v}(n + 1), \mathfrak{K}(n + 1, y)}) \,.
	      \]
	      By construction $B_y$ has length $\tilde{\gamma}_{\mathfrak{v}(n + 1), \mathfrak{K}(n + 1, y)}$.
\end{enumerate}
Note that by the consistency hypothesis on $\tilde{\gamma}$, $(B_x : x \in [d]^n)$ is a partition of $[0, 1)$ for any $n \in \mathbb{N}$, and in fact $(B_x : x \in [d]^{n + 1})$ is a partition nested in $(B_x : x \in [d]^n)$.
Indeed, it is clear that $(B_x : x \in [d])$ is a partition of $[0, 1)$.
Now suppose that $(B_x : x \in [d]^n)$ is a partition of $[0, 1)$.
Note that for a given $x \in [d]^n$, if $y \in [d]^{n + 1}$ is such that $R_n y = x$, then
\[
	B_y = \inf B_{x} + \Bigl[ \sum_{j = 1}^{y_{n + 1} - 1} \tilde{\gamma}_{\mathfrak{v}(n + 1), \mathfrak{K}(n + 1, y + (j - y_{n + 1}) e_{n + 1})}, \sum_{j = 1}^{y_{n + 1}} \tilde{\gamma}_{\mathfrak{v}(n + 1), \mathfrak{K}(n + 1, y + (j - y_{n + 1}) e_{n + 1})} \Bigr) \,,
\]
As a consequence, if $y' \in [d]^{n+1}$ is another vector such that $R_n(y') = x$ and $y' \neq y$, this last equality implies $B_y \cap B_{y'} = \varnothing$.
Now, if for a vector $x \in [d]^n$ and a number $l \in [d]$ we understand $(x, l)$ as a vector in $[d]^{n + 1}$, we obtain that
\begin{align*}
	\bigcup_{y \in [d]^{n + 1} : R_n y = x} B_y
	 & = \bigcup_{l = 1}^d B_{(x, l)}
	= \inf B_x + \bigcup_{l = 1}^d \Bigl[ \sum_{j = 1}^{l - 1} \tilde{\gamma}_{\mathfrak{v}(n + 1), \mathfrak{K}(n + 1, (x, j))}, \sum_{j = 1}^{l} \tilde{\gamma}_{\mathfrak{v}(n + 1), \mathfrak{K}(n + 1, (x, j))}\Bigr) \\
	 & = \inf B_x + \Bigl[ 0, \sum_{j = 1}^d \tilde{\gamma}_{\mathfrak{v}(n + 1), \mathfrak{K}(n + 1, (x, j))} \Bigr)\,.
\end{align*}
By the consistency condition of $\tilde{\gamma}$, $\sum_{j = 1}^d \tilde{\gamma}_{\mathfrak{v}(n + 1), \mathfrak{K}(n + 1, (x, j))} = \tilde{\gamma}_{\mathfrak{v}(n), \mathfrak{K}(n, x)}$, and hence
\[
	\bigcup_{y \in [d]^{n + 1} : R_n y = x} B_y = \inf B_x + [0, \tilde{\gamma}_{\mathfrak{v}(n), \mathfrak{K}(n, x)}) = B_x \,.
\]
This in particular implies that if $y$ and $y'$ are distinct elements in $[d]^{n + 1}$, then $B_y \cap B_{y'} = \varnothing$.
The case where $R_n y = R_n(y')$ was already considered, and the other case follows as we supposed that $(B_x : x \in [d]^n)$ is a partition, $B_y \subset B_{R_n y}$, $B_{y'} \subset B_{R_n(y')}$ and $B_{R_n y} \cap B_{R_n(y')} = \varnothing$.
This proves, by induction, that $(B_x, x \in [d]^{n + 1})$ is a partition of $[0, 1)$, which we have seen that is nested in $(B_x, x \in [d]^n)$.

Using the collection $\mathfrak{B}$ and the random variable $V$ we will define an auxiliary sequence of random variables $(\tilde{Z}_l)_{l \in \mathbb{N}}$ that take values in $[d]$ through their finite dimensional distributions.
For any $n \in \mathbb{N}$ and any $x \in [d]^n$ set $\tilde{Z}_1 = x_1, \ldots, \tilde{Z}_n = x_n$ if and only if $V \in B_x$.
With this definition we note that
\begin{equation}\label{aux_2}
	\mathbb{P}'\Bigl( \bigcap_{l = 1}^n \{\tilde{Z}_l = x_l\} \Bigr) = \tilde{\gamma}_{\mathfrak{v}(n), \mathfrak{K}(n, x)} \,.
\end{equation}
Finally, we define the array $Z$ in terms of the sequence $\tilde{Z}$ by putting
\[
	Z_{l\ell} = \tilde{Z}_{l + (\ell - 1) d} \quad\text{for all $l \in [d]$ and $\ell \in \mathbb{N}$.}
\]
By using the consistency hypothesis of $\tilde{\gamma}$ together with the additivity of $\mathbb{P}'$ it is straightforward to check that $Z$ satisfies \eqref{eq:exchArray1}.

%%%%%%%%%%%%%%%%%%%%%%%%%%%%%%%%%%%%%%%%%%%%%%%%%%%%%%%%%%%%%%%%%%%%%%%%%%%%%%%%
% Bibliography
%%%%%%%%%%%%%%%%%%%%%%%%%%%%%%%%%%%%%%%%%%%%%%%%%%%%%%%%%%%%%%%%%%%%%%%%%%%%%%%%
%\nocite{*}
\bibliographystyle{amsalpha}
\bibliography{bibi.bib}

\end{document}